\documentclass[preprint]{imsart}

\usepackage{ulem,cancel} 
\usepackage{lipsum}
\usepackage{amsfonts,amssymb}
\usepackage{graphicx}
\usepackage{epstopdf}
\usepackage{algorithmic}

\usepackage{amsmath, amssymb, amsthm}
\usepackage{natbib}
\usepackage{hyperref}
\startlocaldefs
\theoremstyle{plain}

\newtheorem{theorem}{Theorem}[section]
\newtheorem{lemma}[theorem]{Lemma}
\newtheorem{corollary}[theorem]{Corollary}
\newtheorem{proposition}{Proposition}[section]
\theoremstyle{remark}

\newtheorem{example}{Example}
\newtheorem{remark}{Remark}

\renewcommand{\kappa}{\varkappa}
\ifpdf
  \DeclareGraphicsExtensions{.eps,.pdf,.png,.jpg}
\else
  \DeclareGraphicsExtensions{.eps}
\fi

\usepackage{enumitem}
\setlist[enumerate]{leftmargin=.5in}
\setlist[itemize]{leftmargin=.5in}

\renewcommand{\H}{\mathcal{H}}

\newcommand{\R}{{\mathbb{R}}}
\newcommand{\N}{{\mathbb{N}}}

\newcommand{\A}{\mathscr{A}}
\newcommand{\B}{\ty{\mathcal{B}}}

\renewcommand{\A}{\mathcal{A}}

\renewcommand{\P}{{\mathsf P}}

\newcommand{\E}{{\mathsf E}}

\renewcommand{\i}{\mathrm{i}}

\newcommand{\C}{\mathbb{C}}
\newcommand{\CC}{\mathcal{C}}

\def\R{\mathbb{R}}
\def\N{\mathbb{N}}

\newcommand{\eps}{\varepsilon}

\newcommand{\Var}{\operatorname{Var}}
\newcommand{\cov}{\operatorname{cov}}

\newcommand{\argmin}{\operatornamewithlimits{arg\,min}}

\renewcommand{\Re}{\operatorname{Re}}

\renewcommand{\L}{\mathcal{L}}

\renewcommand{\i}{\mathrm{i}}
\renewcommand{\B}{\mathcal{B}}

\endlocaldefs

\begin{document}

\begin{frontmatter}
\title{{Decompounding under general mixing distributions}}
\runtitle{{Decompounding under general mixing distributions}}

\begin{aug}

\author[A]{\fnms{Denis}~\snm{Belomestny}\ead[label=e1]{denis.belomestny@uni-due.de}\ead[label=u1,url]{https://denbel.github.io}\orcid{0000-0002-9482-6430}}
\author[B]{\fnms{Ekaterina}~\snm{Morozova\!\!}\ead[label=e2]{ekaterina.morozova@uni-due.de}
\orcid{0009-0001-1024-512X}}
\author[C]{\fnms{Vladimir}~\snm{Panov}\ead[label=e3]{vpanov@hse.ru}\orcid{0000-0001-8395-1909}}
\address[A]{University of Duisburg-Essen, Essen, Germany\printead[presep={,\ }]{e1,u1}}

\address[B]{University of Duisburg-Essen, Essen, Germany\printead[presep={,\ }]{e2}}

\address[C]{HSE University, Moscow, Russia\printead[presep={,\ }]{e3}}

\runauthor{Belomestny et al.}
\end{aug}

\begin{abstract}
This study focuses on statistical inference for compound models of the form \(X=\xi_1+\ldots+\xi_N\), where \(N\) is a random variable denoting the count of summands, which are independent and identically distributed (i.i.d.) random variables \(\xi_1, \xi_2, \ldots\). The paper addresses the problem of reconstructing the distribution of \(\xi\) from observed samples of \(X\)'s distribution, a process referred to as decompounding, with the assumption that \(N\)'s distribution is known. This work diverges from the conventional scope by not limiting \(N\)'s distribution to the Poisson type, thus embracing a broader context. We propose a nonparametric estimate for the density of  \(\xi,\)  derive its  rates of convergence and prove that these rates are minimax optimal for suitable classes of distributions for \(\xi\) and \(N\). Finally, we illustrate the numerical performance of the algorithm on simulated examples.
\end{abstract}

\begin{keyword}
\kwd{Compound models}
\kwd{decompounding}
\kwd{inverse problem}
\kwd{minimax rates}
\end{keyword}

\end{frontmatter}

\section{Introduction}
Consider a random variable \(X\) defined as a sum of a random number \(N\) of independent and identically distributed (i.i.d.) random variables \(\xi_1, \xi_2, \ldots\), i.e.,
\begin{equation}
\label{randomsum_main}
X = \sum_{k=1}^{N} \xi_k, 
\end{equation}
where \(N\) and \(\xi_1, \xi_2, \ldots\) are independent. This model can be seen as a generalization of Poisson random sums, which corresponds to the case when \(N\) follows the Poisson distribution. Given a sample from this model, a natural question is how to estimate the distribution of \(\xi_1\), assuming that the distribution of \(N\) is known. This problem has been explored in several studies, primarily within a parametric framework and especially when \(N\) is Poisson-distributed. A critical observation in nearly all these studies is that the characteristic function of \(X\) equals the superposition of the probability generating function of \(N\) and the characteristic function of \(\xi\),
\[ 
\phi_X(u) = \mathsf{E} [e^{\i u X}] = \mathcal{P}_N (\phi_\xi(u)) \quad \text{with} \quad \mathcal{P}_N(z) = \sum_{k=1}^\infty \mathtt{p}_k z^k, \quad \mathtt{p}_k=\P(N=k),
\]
where \(\phi_\xi(u) = \mathsf{E} [e^{i u \xi}]\) is the characteristic function of \(\xi\). Therefore, the function  \(\phi_\xi(u)\) (and hence the distribution of \(\xi\)) can be recovered if the inverse function of \(\mathcal{P}_N\) is well-defined and precisely known. In simpler cases, such as when \(N\) has a Poisson or geometric distribution, this process is straightforward.  In the case when \(\xi\) is a discrete random variable, one can use the recursion formulas (known as the Panjer recursions in the context of actuarial calculus, see Johnson et al., \citeyear{JKK}) to recover the distribution of  \(\xi\).  For other methods of this type, we refer to the monograph by Sundt and Vernic~(\citeyear{SundtVernic:09}). However, the problem is inherently ill-posed; minor perturbations in \(\phi_X\) (e.g., due to limited data) can result in significant errors in \(\phi_\xi(u)\).
In addressing the statistical  aspect of this problem, the concept of Panjer recursion was utilized by Buchmann and Gr\"ubel~(\citeyear{BuchmannGrubel:03}, \citeyear{BuchmannGrubel:04}), who are credited with introducing the term ``decompounding''. Subsequently, other estimation methods were developed for the same model, including  kernel-based estimates (Van Es et al., \citeyear{VanEs_decompounding}), convolution power estimates (Comte et al., \citeyear{Comte_CPP}), spectral approach (Coca, \citeyear{Coca_decompounding}) and Bayesian estimation techniques (Gugushvili et al., \citeyear{Gugushvili2018non}). Let us stress that all  these methods are designed   for the case when \(N\) has  a Poisson distribution. 
The case of a generally distributed \(N\) was examined by B{\o}gsted and Pitts~(\citeyear{Bogsted}), who proposed inverting \(\mathcal{P}_N\) via series inversion. However, they did not provide convergence rates, and their approach requires that \(\xi_1\) be positive with probability \(1\) and \(N\) take the value 1 with positive probability.
\par
It's worth noting that there's a significant demand for broader classes for the distribution of  
\(N\)  across different applied disciplines, such as actuarial science and queuing theory. Asmussen and Albrecher~(\citeyear{AA_ruin}) suggested that the widespread use of the Poisson distribution in actuarial models is largely attributed to its analytical simplicity and the ease with which its results can be interpreted, rather than any concrete evidence of its efficacy. Meanwhile, the exploration of non-Poisson arrival processes in queuing theory has been the focus of numerous studies. For instance, the work by Chydzinski~(\citeyear{chydzinski2020queues}) delves into these processes, offering insights that can be contrasted with findings from Poisson-based models, as seen in the paper by Hansen and Pitts~(\citeyear{HansenPitts}) or den Boer and Mandjes~(\citeyear{den2017convergence}).\newline
\par

\textit{Contribution.} This paper makes two significant contributions. Firstly, we conduct  an in-depth theoretical analysis of the underlying statistical inverse problem when \(N\), the count variable, has a general distribution. Our approach leverages the equation
\begin{equation} 
\label{eq:NN}
\phi_X(u) = \mathcal{L}_N\bigl(-\psi_\xi(u)\bigr), 
\end{equation}
where \(\mathcal{L}_N(w) = \mathsf{E}[e^{-wN}]\) for \(w\in\mathbb{C}\), is the Laplace transform of \(N\), and \(\psi_\xi(u) = \log(\phi_\xi(u))\), assuming the principal branch of the complex logarithm and that the characteristic function of \(\xi\) is devoid of real zeros. The method involves estimating \(\psi_\xi\) by inverting \(\mathcal{L}_N(w)\) with respect to \(w\). Following this, we apply the regularized inverse Fourier transform to approximate the density of \(\xi\).
Secondly, the paper establishes convergence rates for the proposed density estimate across various distribution classes, demonstrating that these rates achieve minimax optimality. Notably, we find that when \(\mathsf{P}(N=1) > 0\), the minimax convergence rates align with those obtained from direct observations of \(X\). This marks the first instance of deriving minimax rates for general \(N\) scenarios in existing literature.\newline

\textit{Structure.} The paper is organised as follows. In the next section we present the estimation procedure and show the rates of convergence to the true density. Subsection~\ref{sec32} is devoted to the nonasymptotic properties of the proposed estimate. The main result of this section, Theorem~\ref{thm:gen-bound}, gives the upper bound of the MSE on the sequence of events \(\A_n\), the probabilities of which tend to 1 rather fast provided that some conditions are fulfilled (see  Lemma~\ref{cor-sym} and examples in Subsection~\ref{sec22}). Next, in Subsection~\ref{sec23}, we show the asymptotic upper bounds for several classes of distributions  (Theorem~\ref{thm:minmax-pol}) and prove that these bounds are minimax optimal for several classes of densities (Theorem~\ref{thm:lb}).  Also we propose an adaptive procedure for the choice of the cutoff parameter, and show that the MSE of our estimate with adaptive choice differs from the minimal possible MSE by a negligible term, at least under some assumptions on the considered class (Theorem~\ref{thm25}). Section~\ref{sec3} contains several numerical examples showing the efficiency of the proposed algorithm and illustrating the difference between various classes of distributions.  Next, in Section~\ref{sec:RD} we provide a real-life example of the recovery of  the claim amount density on the  liability policies dataset. In Section~\ref{sec4}, we discuss the case of constant \(N\).  The main outcomes from our study and some related open problems are discussed in Section~\ref{sec_con}. The proofs are collected in Section~\ref{sec5}.

\section{Main results}
\label{sec:main}
The key idea of the estimation procedure is to apply the inverse  Laplace transform of \(N\) (with respect to its complex-valued argument) to both parts of the equality~\eqref{eq:NN}, that is, 
\begin{eqnarray}
\label{Ninverse}
\psi_\xi(u) = -\L^{-1}_N(\phi_X (u)), \quad u\in\R.
\end{eqnarray}
Note that \(\L_N(w)\) is an analytic function for any \(w\) with \(\Re(w)>0\) (in particular, for \(w=-\psi_\xi(u), \;  u \in \R\)), and therefore  the inverse   function \(\L^{-1}_N\) exists and is analytic at the point \(\phi_X (u)\), provided that
\begin{eqnarray}\label{LNN}
(\L_N)'\left(-\psi_\xi(u)\right) \ne 0,\quad u \in \R.
\end{eqnarray}
Note that if \(\E[N]<\infty,\) we have  
\begin{eqnarray*}
(\L_N)'\left(-\psi_\xi(u)\right) 
=   \sum\limits_{k=1} ^{\infty} k (\phi_\xi(u))^{k}\mathtt{p}_k
= \E [N] \phi_\Lambda(u),
\end{eqnarray*}
where \(\phi_\Lambda(u)\) is the characteristic function of the random variable  \(\Lambda = \xi_1 + ... + \xi_\tau\) with \(\tau\) such that \(\P(\tau=k) = k \mathtt{p}_k/\E[N], k=1,2,\dots\), and therefore~\eqref{LNN} is equivalent to \(\phi_\Lambda(u) \ne 0, \; u \in \R.\) This assumption holds under rather simple sufficient conditions (see Appendix~\ref{appB}). We won't be discussing these conditions here as we need to generalise~\eqref{LNN} to determine the convergence rates in the next section. More precisely, we will assume that \((\L_N)'(z)\neq 0\) not only at the point $-\psi_\xi(u)$, but also in some vicinity of this point, which we will specify later. 

The formula~\eqref{Ninverse} suggests  the following estimation scheme. First, we estimate the characteristic function \(\phi_X (u)\) based on a sample \(X_1,\ldots,X_n\) from the distribution of $X$ using the empirical characteristic function \(\widehat{\phi}_{X}(u)=n^{-1} \sum_{k=1}^n e^{\i u X_k}. \)  Second, we estimate the function \(\psi_\xi\) via  \(\widehat{\psi}_\xi(u)=-\L^{-1}_N(\widehat\phi_X (u))\) and get  estimate for the characteristic function of \(\phi_\xi\) as \(\widehat\phi_\xi(u)=\exp(\widehat{\psi}_\xi(u)).\)
Note that \(\L^{-1}_N(\widehat\phi_X (u))\) is well defined for \(u\) in some vicinity of \(0\) due to \(\Re(\widehat\phi_\xi(0))=1.\)  
Finally, we 
use a regularised  inverse Fourier transform to estimate the distribution of \(\xi. \)  So, the scheme is as follows: 
\[ \boxed{X_1,\dots, X_n \; \rightarrow \; \widehat\phi_X (u)  
\; \rightarrow \; 
\widehat{\psi}_\xi (u)
\; \rightarrow \; 
\widehat{\phi}_{\xi} (u)
\rightarrow \;
\text{Law}(\xi). }\]
Assuming that the distribution of \(\xi\) is absolutely continuous with density \(p_\xi\),  the estimation scheme explained above leads to the following estimator
\begin{equation}
\label{est2}
\widehat{p}_\xi(x) := \frac{1}{2\pi} \int_{-U_n} ^{U_n} e^{-\i ux}\widehat{\phi}_{\xi} (u)\,du 
= \frac{1}{2\pi}\int_{-U_n} ^{U_n} \exp\left\{-\i ux - \L^{-1}_N\bigl(\widehat{\phi}_X(u)\bigr)\right\}\,du
\end{equation}
for any \(x\in\R,\) where \(U_n\) is a sequence of positive numbers  tending to infinity at a proper rate in order to ensure that \(\L^{-1}_N\bigl(\widehat{\phi}_X(u)\bigr)\)  for \(u\in[-U_n,U_n]\) is well defined on an event of positive probability.

Note that the estimate~\eqref{est2} is real-valued, because
\begin{eqnarray*}
\overline{\widehat{p}_\xi(x)} &:=& \frac{1}{2\pi}\int_{-U_n} ^{U_n} \exp\left\{\i ux - \overline{\L^{-1}_N\bigl(\widehat{\phi}_X(u)\bigr)}\right\}\,du\\
&=&
\frac{1}{2\pi}\int_{-U_n} ^{U_n} \exp\left\{\i ux - \L^{-1}_N\bigl(\widehat{\phi}_X(-u)\bigr)\right\}\,du = \widehat{p}_\xi(x),
\end{eqnarray*}
where the second equality follows from the fact 
\(
\overline{\L^{-1}_N(z)}  = \L^{-1}_N(\overline{z}).
\)
Indeed, 
\begin{eqnarray*}
\L_N\bigl( 
\overline{\L^{-1}_N(z)}
\bigr)  =
\sum_{k=1}^\infty e^{-\overline{\L^{-1}_N (z)} k} \P \Bigl\{ 
N=k
\Bigr\} = \overline{\sum_{k=1}^\infty e^{-\L^{-1}_N (z) k} \P \Bigl\{ 
N=k
\Bigr\}}
=\overline{z}.
\end{eqnarray*}

\subsection{Nonasymptotic bounds}\label{sec32}

As we will see below, the study of the difference between the estimator~\eqref{est2} and the true density function \(p_\xi(u)\) is based on the Taylor decomposition of the function  
\begin{eqnarray*}
\H(z):=\exp(-\L^{-1}_N(z)),
\end{eqnarray*}
in the vicinity of the point \(z= \phi_X(u)\), which covers the point \(z=\widehat\phi_X(u).\) This vicinity should be included in
  the region of analyticity of the function \(\H(z)\), which we denote by \(\CC\). Moreover, for the theoretical study, we need a boundedness of \(\H'(z)\) and \(\H''(z)\) in this vicinity. 
 
Below we will introduce the event \(\mathcal{A}_n(\varkappa)\), which reflects these properties. However,  for the sake of clarity, let us show that these properties hold at the point \(z=\phi_X(u)\) under some mild conditions. In fact,   the discussion above implies that \(\phi_X(u) \in \CC, \forall u \in \R,\) provided that  \(\phi_\Lambda\) does not vanish on \(\R\).  The direct calculation yields \begin{eqnarray}\label{H}
\H'(z) &=& - \frac{\exp(-\L^{-1}_N(z))}{\left(\L_N\right)'\left(\L^{-1}_N(z)\right)} = 
 \frac{1}{\sum_{k=1}^\infty k \texttt{p}_k e^{-(k-1)\L^{-1}_N(z) }}
, \qquad z \in \CC,\\
\label{H2}
\H''(z)&=&
(\H'(z))^3 \sum_{k=2} ^{\infty} \mathtt{p}_k (k^2-k)\exp(-(k-2)\L^{-1}_N(z)), \qquad z \in \CC.
\end{eqnarray}
At the point \(z=\phi_X(u), u \in \R\), these derivatives are equal to 
\begin{eqnarray*}\label{G}
\H'\bigl (\phi_X(u)\bigr) 
&=& - \frac{\phi_{\xi}(u)}{(\L_N)'(-\psi_{\xi}(u))}
=
-
 (\E[N])^{-1}
\frac{\phi_{\xi}(u)}{\phi_{\Lambda}(u)}, \\ 
\H''\bigl (\phi_X(u)\bigr) 
&=&
 (\H'(\phi_{X}(u)))^3 \sum_{k=2} ^{\infty} \mathtt{p}_k (k^2-k)\bigl( 
 \phi_\xi(u)
 \bigr)^k,\end{eqnarray*} 
and therefore are uniformly bounded on \(\R\), provided that \(\E[N^2]<\infty, \mathtt{p}_1 >0\) and \eqref{LNN} holds. Indeed, \(\phi_\Lambda\) does not vanish on \(\R\), and in this case  \(\H'\bigl (\phi_X(u)\bigr) \) is bounded as a continuous function   since its limit  for \(u \to \infty\) is equal to 
\begin{eqnarray*}
\lim_{u\to\infty} \H'\bigl (\phi_X(u)\bigr) &=&
- \frac{
 (\E[N])^{-1}}{
 \P\bigl\{\tau=1\bigr\} + \lim_{u\to\infty} \sum_{k=2}^\infty
 (\phi_\xi(u))^{k-1} \P\bigl\{\tau=k\bigr\}}\\
&=&
-
 \frac{
 (\E[N])^{-1}}{
  \P\bigl\{\tau=1\bigr\} 
  }  =
  -
  \frac{1}{\mathtt{p}_1},
\end{eqnarray*}
where the second equality follows from the Riemann--Lebesgue lemma and the fact that the series is uniformly bounded by 1. 
Moreover, \( \H''\bigl (\phi_X(u)\bigr) \) is bounded due to the trivial estimate \(| \H''\bigl (\phi_X(u)\bigr) | \leq  | \H'\bigl (\phi_X(u)\bigr) |^3(\E[N^2]+\E[N]), u \in \R.\)

Now let us turn towards the analyticity of \(\H(z)\) and the boundedness of \(\H'(z)\) and \(\H''(z)\) in some vicinity of the point \(z=\phi_X(u)\). Introduce the following event 
\begin{multline*}
\mathcal{A}_n(\varkappa):=
\biggl\{ 
\phi_{X,\tau}(u) \in \CC,  \; \forall \tau\in [0,1], \; \forall u\in[- U_n,U_n]
\biggr\} \\
\cap
\biggl\{\max_{\tau\in [0,1]}\max\limits_{|u|\leq U_n} \max\bigl\{ |\H'\bigl(\phi_{X,\tau}(u)\bigr)|, |\H''\bigl(\phi_{X,\tau}(u)\bigr)|\bigr\} \leq \varkappa \biggr\},
\end{multline*}
where \(\phi_{X,\tau}(u):=\phi_X(u)+\tau(\widehat{\phi}_X(u)-\phi_X(u))\) and \(\varkappa>0.\) Note that if this event has positive probability for all \(n\), then the assumption~\eqref{LNN} holds. When there is no risk of confusion, we will use the simplified notation \(\A_n:=\A_n(\varkappa)\).

Let us formulate the main result of this section.
\begin{theorem}\label{thm:gen-bound}
Suppose that \(\phi_\xi\in L_1(\R)\) with \(C_\phi:=\|\phi_\xi\|_{L^1(\R)}.\)
Let us fix some \(\varkappa>0\) such that \(q_n=q_n(\varkappa):=\P(\mathcal{A}_n (\kappa))>0\) for all $n>n_0.$ 
Assume also that   \(\E [N^2]<\infty\), 
then it holds 
\begin{multline}
\label{thm2bound}
\E\left[\left(\widehat{p}_\xi(x)-p_\xi(x)\right)^2\Bigr|\A_n\right] 
\lesssim \biggl(\int_{|u|>U_n} \left|\phi_{\xi}(u)\right|\,du\biggr)^2 
+ \varkappa^2 \frac{U_n^2}{n^2 q_n^2}  \\+ \varkappa^2\frac{U_n }{n\,q_n}   C_\phi \E[N] +\varkappa^2\frac{U_n^2 }{n\,q_n}  (1-q_n)^{1/2} 
+
\varkappa \frac{U_n^2 }{n^{3/2}\,q_n^{1/2}} , \qquad n>n_0, 
\end{multline}
where \(\lesssim\) stands for inequality up to an absolute constant not depending on \(n\) and distributions of \(\xi, N.\)
\end{theorem}
\begin{proof} The proof relies on the bias-variance decomposition, see Section~\ref{sec5}.  The first two terms in the right-hand side of~\eqref{thm2bound} correspond to the bias, while the last three terms correspond to the variance.
\end{proof}
The definition of the event \(\mathcal{A}_n\) is based on the function 
 \(\phi_{X,\tau}(u)\),
which depends on the unknown distribution \(\xi\) twice: through the characteristic function \(\phi_X\) and its estimator \(\widehat{\phi}_X\). This dependence makes the analysis of \(q_n=\P\{\mathcal{A}_n\}\) quite complicated, as there is even no guarantee that these probabilities are monotonically increasing with \(n\). However, the next lemma provides several cases, when the situation becomes simpler, and  \(q_n\) (or lower bounds for \(q_n\)) are the same for any distributions of \(\xi\) from certain classes. 


\begin{lemma}\label{cor-sym}
\begin{enumerate}
\item[(i)] If the distribution of \(\xi\) is symmetric, then there exists some \(\kappa>0\) such that  \(q_n=\P(\A_n(\varkappa))=1\) for all \(n\geq 1.\) 
\item[(ii)] If  there is \(\rho_0>1\) and some finite \(\varkappa=\varkappa(\rho_0)\) such that
\begin{equation}
\label{eq:ex-ass-h}
|\H'(z)|\leq \varkappa,\quad |\H''(z)|\leq \varkappa, \qquad \forall \; z\in \C: |z|\leq \rho_0,
\end{equation}
then also \(q_n=1\) for all \(n\geq 1.\) 
\item[(iii)]
If  there is \(\rho_1>1\) and some finite \(\varkappa=\varkappa(\rho_1)\) such that
\begin{equation}
\label{eq:ex-ass-h-phi}
|\H'(z)|\leq \varkappa,\quad |\H''(z)|\leq \varkappa,  \qquad \forall \; z \in \C: \Re(z)> -\rho_1,
\end{equation}
and moreover \(\Re(\phi_X(u))\neq 0 \;\; \forall u \in \R,\) then \(1-q_n \lesssim  \left(\sqrt{n}U_n\right)^{-2}\) for any  \(n\) such that 
\begin{equation}\label{nUn}
\log (n U_n^{2}) / n<(\rho_1/18)^2. 
\end{equation}
\end{enumerate}
\end{lemma}

\begin{proof}
(i)  Suppose that the distribution of \(\xi\) is symmetric, then by changing the empirical characteristic function
\(\widehat\phi_X(u)\) to its real part  we have that \(\phi_{X,\tau}(u)\in [-1,1]\) for all \(u \in \R.\) Hence, with probability \(1,\)
\[
| \H'(\phi_{X,\tau}(u)) | = \left[\sum_{k=1}^{\infty} k\, \mathtt{p}_k \exp\bigl(-(k-1)\L^{-1}_N(\phi_{X,\tau}(u))\bigr)\right]^{-1}\leq \frac{1}{\mathtt{p}_1},\quad \tau \in [0,1],
\]
for any \(u\in\R\), provided \(\mathtt{p}_1>0.\) Furthermore, we have  that \(\L^{-1}_N(\phi_{X,\tau}(u))\geq 0\) for all \(u\in \mathbb{R}\) and \(\tau\in [0,1]\) with probability \(1\) due to the fact that 
\[
\L_N(z) = 
\sum_{k=1}^{\infty} \mathtt{p}_k e^{-zk}\in [0,1] \quad \text{iff} \quad z\geq 0.
\]
Therefore, from~\eqref{H2} we get 
\begin{eqnarray*}
|\H''(\phi_{X,\tau}(u))|&\leq& \frac{\E[N^2]-\E[N]}{\mathtt{p}_1^3}=:\varkappa,\end{eqnarray*}
and conclude that under this choice of \(\varkappa\), we have \(q_n=\P\{\A_n(\varkappa)
\}=1\) for all \(n\geq 1.\) 

(ii)  This statement is obvious, since \(|\phi_{X,\tau}(u)|\leq 1\) for all \(u \in \R\) and \(\tau \in [0,1].\)

(iii) Using the fact that \(\Re(\phi_X(u))\geq 0\) for all \(u\in \R\) and Proposition 3.3 from Belomestny et al., \citeyear{BR_Levy}, we derive
\begin{eqnarray*}
1-q_n=\P\{\A_n^c\}&\leq& \P(\exists \tau \in [0,1], \exists u\in [-U_n,U_n]: \Re(\phi_{X,\tau}(u))\leq -\rho_1) \\ 
&\leq&  \P\bigl(\|\phi_X-\widehat \phi_X\|_{[-U_n,U_n]}\geq \rho_1\bigr)
\lesssim\left(\sqrt{n}U_n\right)^{-2},
\end{eqnarray*}
provided that~\eqref{nUn} holds.
\end{proof}
\subsection{Examples}\label{sec22}
In this section we discuss several important examples including two-point distribution and Poisson like distribution for \(N.\)
\begin{example}
\label{exm:two-points}
Let \(N\) take two values, 1 and 2, with probabilities \(p\in(0,1)\) and \(1-p\), respectively. Then
\begin{eqnarray} \label{LNex4}
\L_N(z) = pe^{-z}+(1-p)e^{-2z},\quad \forall z\in\C.
\end{eqnarray}
The inversion of the Laplace transform~\eqref{LNex4} leads to 
\begin{eqnarray*}
\L_N^{-1}(z) &=& -\log\Bigl( 
\frac{-p + \sqrt{p^2 + 4z(1-p)}}{2(1-p)}
\Bigr), 
\end{eqnarray*}
and therefore the function \(\H\) and its first and second derivatives are equal to
\begin{eqnarray*}
\H(z) 
&=& \frac{-p + \sqrt{p^2 + 4z(1-p)}}{2(1-p)},\\ 
\H'(z) 
&=&  \frac{1}{\sqrt{p^2 + 4z(1-p)}}, \\ 
\H''(z) 
&=& -   \frac{2(1-p)}{\bigl(p^2 + 4z(1-p)\bigr)^{3/2}}.
\end{eqnarray*}
If \(\rho^*:=p^2/ (4(1-p))>1\) the condition \eqref{eq:ex-ass-h} holds with \(1<\rho_0<\rho^*\) and 
\begin{eqnarray}\label{varkappa0}\varkappa=\varkappa(\rho_0):=\max\Bigl\{ \H'(-\rho_0), -\H''(-\rho_0)
\Bigr\}.\end{eqnarray}
On the other hand, the condition \eqref{eq:ex-ass-h-phi} is also fulfilled for \(0<\rho_1<\rho^*\) and \(\varkappa=\varkappa(\rho_1)\). 
As compared to \eqref{eq:ex-ass-h},  we do not need to assume that \(\rho^*>1,\) but have to check the condition~ \(\Re(\phi_X(u))\neq 0 \;\; \forall u \in \R,\) which depends on the distribution of \(\xi\).
\end{example}
\begin{example}
\label{exm:poisson}
Let \(N\) be distributed according to the shifted Poisson law, that is,
\[
\P\{N=k\}=\mathtt{c}_\lambda \frac{\lambda^k}{k!},\quad k=1,2,\dots,
\]
where \(\lambda>0\) and \(\mathtt{c}_\lambda:= 1/ (e^{\lambda}-1)\). In this case,
\begin{eqnarray*}
\L_N(z) = \mathtt{c}_\lambda \bigl( 
e^{\lambda e^{-z}} 
-1
\bigr),\qquad 
\L^{-1}_N(z) = -\log\Bigl(\frac{1}{\lambda}\log\left(z/\mathtt{c}_\lambda+1\right)\Bigr),
\end{eqnarray*}
where the formula for the inverse function is valid for all \(z\in\C\setminus\R_{\leq 0}\). 
Consequently direct calculations yield
\begin{eqnarray*}
\H(z) &=&  \frac{1}{\lambda}\log\left(z/\mathtt{c}_\lambda+1\right),\\
\H'(z) &=&  \frac{1}{\lambda \bigl( \mathtt{c}_\lambda+z\bigr)},\\
\H''(z) &=&- \frac{1}{\lambda \bigl(\mathtt{c}_\lambda+z\bigr)^2}.
\end{eqnarray*}

Again the condition   \eqref{eq:ex-ass-h} holds with \(1<\rho_0<\mathtt{c}_\lambda\) and \(\varkappa=\varkappa(\rho_0)\) given by~\eqref{varkappa0} with \(\H', \H''\) as above, provided \(\mathtt{c}_\lambda>1.\)  Similarly, the condition
\eqref{eq:ex-ass-h-phi} is also fulfilled for \(0<\rho_1<\mathtt{c}_\lambda\) and \(\varkappa=\varkappa(\rho_1).\)
\end{example}
\begin{example}
\label{exm:geom}
Let \(N\) be geometrically distributed with parameter \(p\in(0,1)\), that is,
\[
\P(N=k) = (1-p)^{k-1} p, \quad k= 1,2,\ldots.
\]
In this case, one observes that 
\begin{eqnarray*}
\L_N(z) &=& \frac{pe^{-z}}{1-(1-p)e^{-z}},\quad \Re(z)>\log(1-p), \\
\L^{-1}_N(z)
&=& -\log\left(\frac{z}{p+z(1-p)}\right),\quad z\neq -p/(1-p).\end{eqnarray*}
Direct calculations lead to the following formulas for any \(z\neq -p/(1-p):\)
\begin{eqnarray*}
\H(z) &=&
\frac{z}{p+z(1-p)},\\
\H'(z) &=& \frac{p}{(p+z(1-p))^{2}}, \\ 
\H''(z) &=& \frac{-2 p (1-p)}{(p+z(1-p))^{3}},
\end{eqnarray*}
and we conclude that if \(\rho^\star:=p/(1-p)>1\) (that is, \(p>1/2\)) the condition \eqref{eq:ex-ass-h} holds with \(1<\rho_0<\rho^\star\)  and \(\varkappa=\varkappa(\rho_0).\) On the other hand, the condition \eqref{eq:ex-ass-h-phi} is also fulfilled for \(0<\rho_1<\rho^\star\) and \(\varkappa=\varkappa(\rho_1).\) Compared to \eqref{eq:ex-ass-h}, here we do need to assume that \(\rho^\star>1.\)
\end{example}
\subsection{Minimax rates of convergence} \label{sec23}
As we have seen in Theorem~\ref{thm:gen-bound}, the non-asymptotic upper bound of the MSE consists of 5 terms. In general, 
the number of terms cannot be reduced  due to the unknown behaviour of 
 the probabilities \(q_n(\kappa)=\P\{\A_n(\kappa)\},\) which depend on the unknown distribution of \(\xi\). The probabilities may even be non-monotonic with respect to \(n\), see comments after Theorem~\ref{thm:gen-bound}. For our further analysis we will use the assumption \(1-q_n \lesssim U_n^{-2}\), which ensures that the  last two terms in~\eqref{thm2bound} are of order smaller than (or equal to) \(U_n/n,\) while the third term is exactly of this order. Note that this assumption  is satisfied for all cases considered in Lemma~\ref{cor-sym}.

On another side, the upper bound~\eqref{thm2bound} for the error of the proposed estimator essentially depends on the asymptotic behaviour of \(\phi_{\xi}\). Let us consider the following two classes of densities with respect to the behaviour of the corresponding characteristic functions: for some fixed \(\beta, \gamma>0,\) \(M>0\) define
\begin{eqnarray*}
\mathcal{C}(\beta,M):=\biggl\{p\in L_1(\R),\, p\geq 0:\quad \sup_{u\in \R} \Bigl\{(1+|u|)^{1+\beta}|\phi_p(u)|\Bigr\}\leq M\biggr\}
\end{eqnarray*}
and
\begin{eqnarray*}
\mathcal{E}(\gamma,M):=\biggl\{p\in L_1(\R),\,  p\geq 0:\quad \sup_{u\in \R} \Bigl\{\exp(c_\gamma|u|^\gamma)|\phi_p(u)|\Bigr\}\leq M\biggr\}, 
\end{eqnarray*}
where 
\[
\phi_p(u)=\int_{\R} e^{\i ux}p(x)\,dx, \qquad u \in \R,
\]
is the characteristic function of the random variable with the density \(p\). It can be seen that these classes are closely related to the notions of smooth and supersmooth error densities typical for the additive deconvolution problems; see, e.g., Meister (\citeyear{MeisterDeconv}), Section~2.4.1.

Then the following result holds.
\begin{theorem}
\label{thm:minmax-pol}
Suppose that for some \(\varkappa>1\), it holds \(1-q_n \lesssim  U_n^{-2}\)  for all \(n>n_0,\) that is, \(\P(\A_n)\geq 1-U_n^{-2}.\)  

(i) If 
\(p_\xi \in \mathcal{C}(\beta,M)\) for some \(\beta>1/2, M>0\) then
\begin{eqnarray*}
\max_{x\in \R}\E\left[\left(\widehat{p}_\xi(x)-p_\xi(x)\right)^2\Bigr|\A_n\right] 
\lesssim M^2 U_n^{-2\beta}+\varkappa^2 \frac{U_n^2}{n^2 q_n^2}+C(M, \varkappa,N) \frac{U_n}{n}, \,\, n>n_0,
\end{eqnarray*}
where \(C(M, \varkappa,N)= \varkappa^2 M \bigl(1 + \beta^{-1}\bigr)\E[N].\) 
Furthermore, under the choice \(U_{n}=n^{1/(1+2\beta)}\) we get 
\begin{eqnarray*}
\max_{x\in \R}\E\left[\left(\widehat{p}_\xi(x)-p_\xi(x)\right)^2\Bigr|\A_n\right] 
\lesssim \max(M^2, C(M, \varkappa,N))\, n^{-2\beta/(1+2\beta)},\quad n>n_0.
\end{eqnarray*}
(ii) If 
\(p_\xi \in \mathcal{E}(\gamma,M)\) for some \(\gamma>\gamma_\circ\) with \(\gamma_\circ>0\), 
and  some \(\gamma, M>0\) then
\begin{multline*}
\max_{x\in \R}\E\left[\left(\widehat{p}_\xi(x)-p_\xi(x)\right)^2\Bigr|\A_n\right] 
\lesssim M^2U_{n}^{2(1-\gamma)}e^{-2c_\gamma U_{n}^{\gamma}}\\
+
\varkappa^2 \frac{U_n^2}{n^2 q_n^2}+C(M, \varkappa,N)\frac{U_n}{n}, \quad
n>n_0,
\end{multline*}
where \(C(M, \varkappa,N)= \varkappa^2 M \E[N] \max\bigl(1, \Gamma(\gamma^{-1}_\circ+1)\bigr) c_\gamma^{-1/\gamma}.\) 
Under the choice \[U_{n}=(\log n/(2c_{\gamma}))^{1/\gamma}\] we get 
\begin{multline*}
\max_{x\in \R}\E\left[\left(\widehat{p}_\xi(x)-p_\xi(x)\right)^2\Bigr|\A_n\right] 
\lesssim \max(M^2,C(M, \varkappa,N))\, \frac{(\log n)^{\max(1,2(1-\gamma))/\gamma}}{n}, \\
 n>n_0.
\end{multline*}
\end{theorem}
It turns out that in case when the distribution of \(\xi\) is symmetric around zero, these rates are essentially minimax optimal. Namely, it follows from Lemma~\ref{cor-sym} that 
\begin{multline*}
\sup_{p_\xi\in \mathcal{C}(\beta,M)}\max_{x\in \R}\E\left[\left(\widehat{p}_\xi(x)-p_\xi(x)\right)^2\right] 
\lesssim \max(M^2, C(M, \varkappa,N))\, n^{-2\beta/(1+2\beta)},\\  n>n_0,
\end{multline*}
provided that \(\mathtt{p}_1>0.\) As shown in the next theorem, the lower bounds in this case are essentially of the same order.
\begin{theorem}[Lower bounds]
\label{thm:lb}
Let \(\P(N=1)>0\) and let \(\mathrm{Sym}(\R)\) be a class of symmetric functions on \(\R\). 
\begin{itemize}
\item[(i)] If \(p_{\xi}\in \mathcal{C}(\beta,M)\) with some \(\beta,M>0\), then 
\[
\inf_{\widehat{p}_{\xi}}\sup_{p_\xi\in \mathcal{C}(\beta,M)\cap \mathrm{Sym}(\R)}\max_{x\in \R}\E\left[\left(\widehat{p}_\xi(x)-p_\xi(x)\right)^2\right]\gtrsim n^{-2\beta/(1+2\beta)}.
\]
\item[(ii)] If \(p_{\xi}\in \mathcal{E}(\gamma,M)\) with some \(\gamma,M>0\), then 
\[
\inf_{\widehat{p}_{\xi}}\sup_{p_\xi\in \mathcal{E}(\gamma,M)\cap \mathrm{Sym}(\R)}\max_{x\in \R}\E\left[\left(\widehat{p}_\xi(x)-p_\xi(x)\right)^2\right]\gtrsim n^{-1},
\]
\end{itemize}

where infimum is taken over all estimators, that is, measurable functions of \(X_1,\ldots,X_n.\) 
\end{theorem}

It is worth noting that the polynomial rates of convergence are rather typical for the estimation approaches for compound Poisson processes, see, e.g., Comte et al. (\citeyear{Comte_CPP}), Propositions 4.2 and 4.3. A more detailed comparison with previous results, however, is not possible, since our approach is developed for the nonparametric setting, for which the convergence rates have never been studied before. Nevertheless, achieving parametric rates in this setup demonstrates the effectiveness of the proposed estimation procedure. \newline 

At the end of this section, let us mention that the optimal choices of the parameter \(U_n\) presented in Theorem~\ref{thm:minmax-pol}, depend on the parameters \(\beta\) and \(\gamma\), which are not known. For instance, as we have seen, for the class \(\mathcal{C}(\beta,M)\) with \(\beta > 1/2\), our estimate has optimal rate of convergence \(n^{-2\beta/(1+2\beta)}\) (that is, between \(n^{-1/2}\) and \(n^{-1}\)), provided that \(U_n=n^{1/(1+2\beta)}.\)  This observation leads to the introduction of adaptive estimates of \(U_n\),  which do not depend on \(\beta.\)

The techniques for the adaptive choice of parameters are known for the  estimates
which can be represented as \(\sum_{k=1}^n f(X_i)/n\) with some (known) \(f\). This class of estimates include kernel density estimates (Goldenshluger and Lepski, \citeyear{GL2011}), estimation in the density deconvolution problem (Goldenshluger and Kim, \citeyear{GK2021}), linear functional estimation  (Brenner Miguel et al., \citeyear{MCJ2023}). Some generalisations, when the estimate can be represented in the form \(\varphi\left(\frac{\sum_{k=1}^n f(X_i)}{n}\right)\), where \(\varphi\) is a polynomial function, are known, see, e.g., quadratic functional estimation (Neubert et al., \citeyear{NCJ2024})  and estimation for compound Poisson processes (Comte et al., \citeyear{Comte_CPP}).  Note that our estimate~\eqref{est2} has an essentially more complex structure.

Below we present an adaptive estimate of the parameter \(U_n\), assuming that the true density belongs to the class \(\mathcal{C}(\beta,M)\) with \(1/2<\beta < \bar{\beta}\), where \(\bar{\beta}\) is some (known) constant. We  leave a complete investigation of the adaptive approach (without any restrictions on the class  \(\mathcal{C}(\beta,M)\)) as an open question.

Following Brenner Miguel et al., (\citeyear{MCJ2023}), we  define an adaptive choice of the estimate on the equidistant grid, that is, \(U_n = kh \) with fixed \(h>0\) and \(k\) selected  as follows 
\begin{eqnarray}\label{k_hat}
\hat{k} = \hat{k}(x) = \argmin_{k=1..K_n} 
\Bigl\{ 
A(k, x) + V(k)
\Bigr\},
\end{eqnarray}
where \(V(k) = \ell k /n\) is a penalising term, and for all \(x\in\R\),
\begin{eqnarray*}
A(k,x) := \max_{k' = (k+1),...,K_n} \Bigl(\bigl(\widehat{p}_{\xi,k'}(x) - \widehat{p}_{\xi,k}(x)\bigr)^2 - V(k')\Bigr)_+, \qquad k=1,2,..,K_n,
\end{eqnarray*}
and  \(\widehat{p}_{\xi,k}(x) \)  is the estimate~\eqref{est2} with \(U_n= k h\). 
The following result holds. 
\begin{theorem}\label{thm25}
Assume that  \(p_\xi \in \CC(\beta, M)\) for some \(\beta \in (1/2, \bar{\beta}]\) with \(\bar{\beta}>0\) and \(M>0\). Denote \[\zeta = \frac{2\bar\beta}{2\bar\beta +1} \in [1/2,1).\]
Assume that the conditions of Theorem~\ref{thm:minmax-pol} hold, and \(K_n =O(n^{\delta})\) as \(n \to \infty\) for some \(\delta \in (0, 1-\zeta)\). Then  for any \(\ell>\varkappa^2 M \E[N]
h\), 
\begin{multline}\label{adaptive_bound}
\max_{x\in \R}\E\left[\left(\widehat{p}_{\xi,\hat{k}}(x)-p_\xi(x)\right)^2\Bigr|\A_n\right] \\
\lesssim \min_{k=1.. K_n}
\biggl( 
\Bigl(\int_{|u|>U_n} \left|\phi_{\xi}(u)\right|\,du\Bigr)^2 + V(k) \biggr)
 + n^{-\zeta}.
\end{multline}
\end{theorem} 

As in the paper by Comte et al. (\citeyear{Comte_CPP}), adaptive choice of \(k\) realises the best compromise between
the squared bias and the variance (the first and the second terms inside the min). From Theorem~\ref{thm:minmax-pol} (i)  it follows that the minimum is of order \(n^{-2\beta/(2\beta+1)}\), which is larger than \(n^{-\zeta}\) under our choice of the parameter \(\zeta.\) In other words, our adaptive choice of the parameter \(k\) ensures that the MSE differs from the minimal possible MSE by a negligible term. 
\begin{remark}\label{rem}
 Theorem~\ref{thm25} yields that the minimal value of \(\ell, \) for which the adaptive procedure leads to the minimax optimal convergence rates, is equal to   \(\ell>\varkappa^2 M \E[N]
h\). Note that since the distribution of \(N\) is assumed to be known, \(\E[N]\) is also known. Moreover, the choice of the parameter \(\kappa\) depends only on the distribution of \(N\) in all examples, considered in this paper, see Lemma~\ref{cor-sym} and examples in Section~\ref{sec22}. Finally, let us show that the parameter \(M\) can be (roughly) estimated from the data. From the assumption \(p_\xi \in \CC(\beta, M)\) it follows that the density \(p_X\) of \(X\) also belongs to this class, 
because 
\begin{eqnarray*}
|\phi_X(u) | \leq \sum_{k=1}^\infty\mathtt{p}_k |\phi_\xi(u)|^k \leq|\phi_\xi(u)|,
\end{eqnarray*}
and therefore
\begin{eqnarray*}
\sup_{u\in \R} \Bigl\{(1+|u|)^{1+\beta}|\phi_X(u)|\Bigr\}
\leq
\sup_{u\in \R} \Bigl\{(1+|u|)^{1+\beta}|\phi_\xi(u)| \Bigr\}\leq M.
\end{eqnarray*}
Assuming that \(\beta\leq\bar{\beta}\) (see Theorem~\ref{thm25}), we estimate \(M\) by 
\begin{eqnarray}\label{M_est}
\widehat{M} = \sup_{u\in \R} \Bigl\{(1+|u|)^{1+\bar\beta}|\widehat\phi_X(u)|\Bigr\},
\end{eqnarray}
where in practice the supremum can be changed to the maximum on a equidistant grid. With no doubt, the  knowledge of \(\bar{\beta}\) is an additional assumption that can be considered as a limitation of this approach from both theoretical and practical points of view. 
\end{remark}
\section{Numerical examples}\label{sec3}


Let us consider the following three cases of the law of \(N\): the one supported on two points, the geometric distribution starting from 1  or the shifted Poisson. As for \(\xi\), in what follows we consider the case when \(\xi\) has either the Laplace distribution with zero mean and scale equal to one, or the standard normal distribution. In the first case, since \(\phi_{\xi}(u)=(1+u^2)^{-1}\), we have that \(p_{\xi}\in\mathcal{C}(\beta,M)\) with \(\beta=1\) and \(M=1\), while in the second one, as \(\phi_{\xi}(u)=e^{-u^2/2}\), we get that \(p_{\xi}\in\mathcal{E}(\gamma,M)\) with \(c_{\gamma}=1/2\), \(\gamma=2\) and \(M=1\). Observing that both the standard Laplace and standard normal distributions are symmetric about zero, we conclude by Lemma~\ref{cor-sym} that in both cases the conditions of Theorem~\ref{thm:minmax-pol} are satisfied.  

\begin{figure}[t]
\includegraphics[width=1\linewidth]{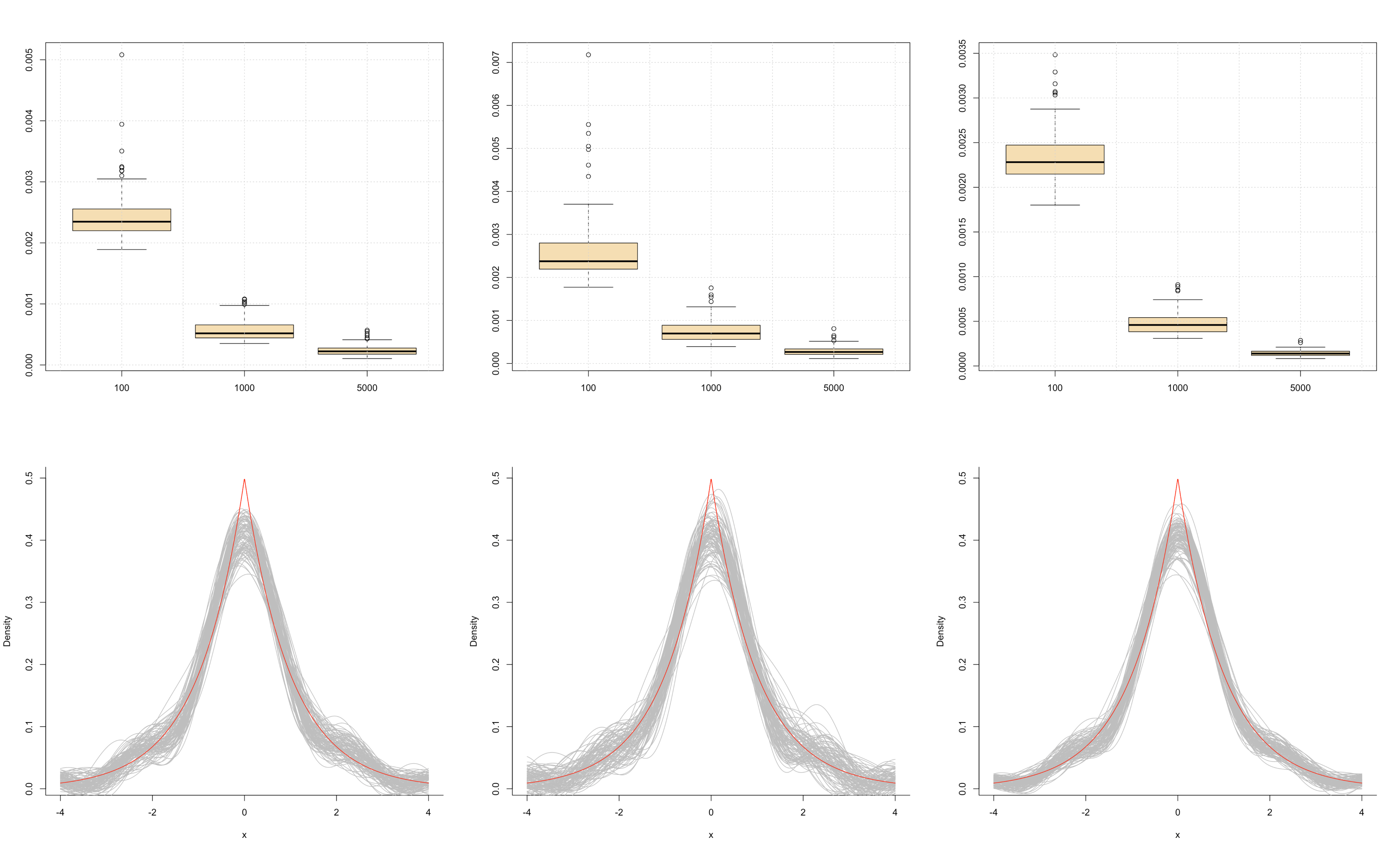}
	\caption{The errors~\eqref{error_def} of the estimator~\eqref{est2} (top) and estimated (grey) and real (red) densities of \(\xi\) (bottom) for \(N\) having the two point (left), geometric (middle) and shifted Poisson (right) distributions, and \(\xi\) following the Laplace distribution.}
	\label{fig:laplace}
\end{figure}

\begin{figure}[t]
\includegraphics[width=1\linewidth]{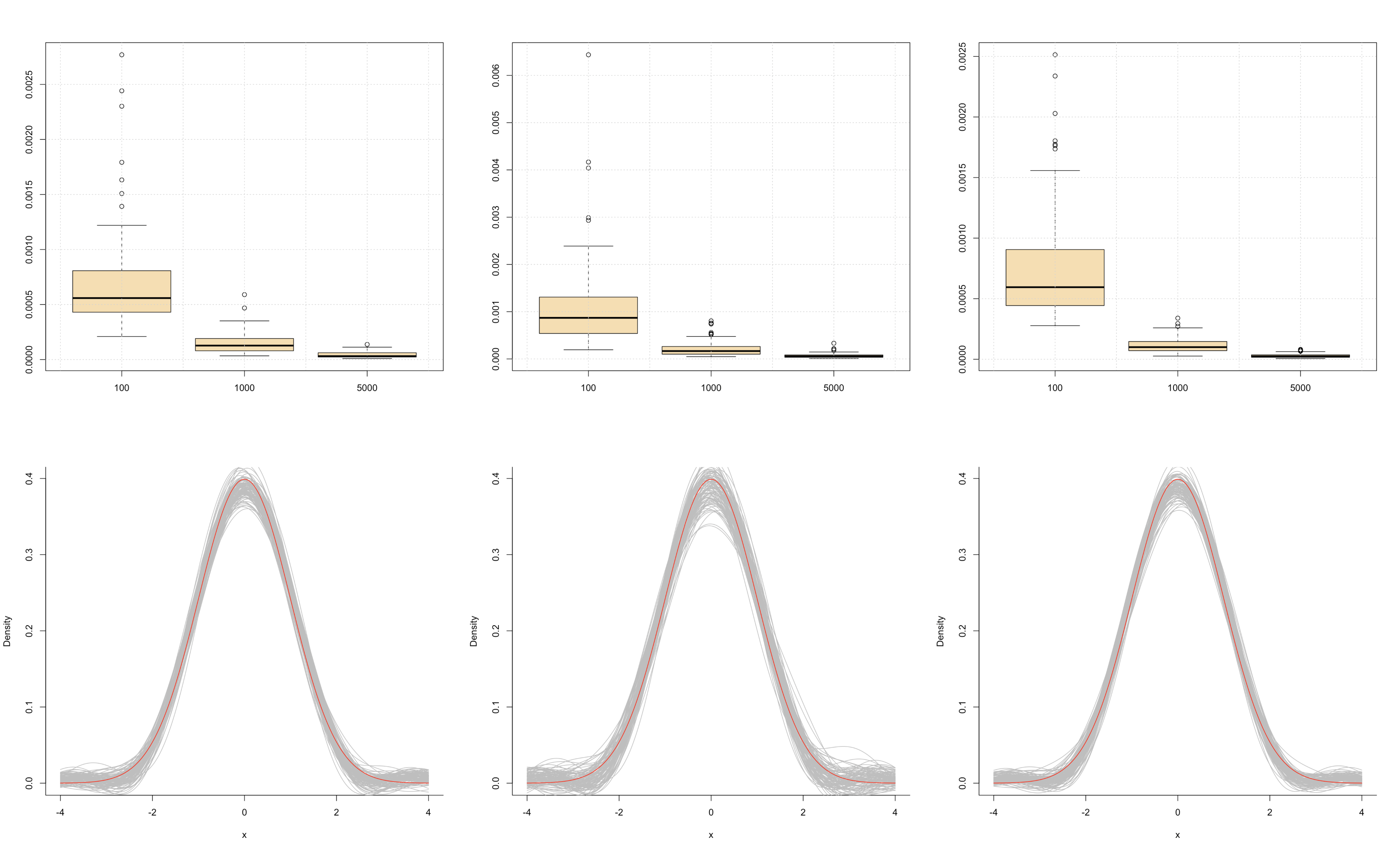}
	\caption{The errors~\eqref{error_def} of the estimator~\eqref{est2} (top) and estimated (grey) and real (red) densities of \(\xi\) (bottom) for \(N\) having the two point (left), geometric (middle) and shifted Poisson (right) distributions, and \(\xi\) following the standard normal law.}
	\label{fig:norm}
\end{figure}

For the simulation study, we fix the parameter of the law of \(N\) as \(p=0.3\) in case of the two-point and geometric distribution and \(\lambda=1\) in case of the Poisson law. For all the considered cases --- three with respect to the distribution of \(N\) and two with respect to the law of \(\xi_1\) --- we aim at analysing the behaviour of the proposed estimator for different values of \(n\in\{100,1000,5000\}\). To this end, we simulate 100 samples for every \(n\) and for each of those samples compute the error
\begin{equation}
\label{error_def}
\frac{1}{J}\sum\limits_{j=1} ^J \left(\widehat{p}_{\xi}(x_j)-p_{\xi}(x_j)\right)^2,
\end{equation}
where \(\{x_j\}_{1\leq j\leq J}\) is an equidistant grid of \(J=1000\) points from \(-4\) to \(4\), \(\widehat{p}_{\xi}\) is the proposed estimator~\eqref{est2}, and \(p_{\xi}\) is the density of either the Laplace or the standard normal distribution. As suggested by Theorem~\ref{thm:minmax-pol}, in case when \(\xi_1\) is normally distributed the truncating sequence is chosen as 
\[
U_n=((\log n )/(2c_\gamma))^{1/\gamma}.
\] 
As for the case of the Laplace distribution, Theorem~\ref{thm:minmax-pol} suggests that \(U_n\) should be of order \(n^{1/(1+2\beta)}=n^{1/3}\). To speed up the numerical computations and ensure convergence of the integrals on finite data samples, we multiply this value by a normalising constant \(c=1/3\), taking \(U_n=cn^{1/3}\).

The first row of Figures~\ref{fig:laplace} and~\ref{fig:norm} represents the boxplots of errors~\eqref{error_def} for the cases of Laplace and normal \(\xi\), respectively, with different distributions of \(N\) and sample sizes \(n\). It can be observed that in all the considered cases the values of errors decline with the growth of sample size and are reasonably small, not exceeding 0.008 even for \(n=100\). Also, it can be seen that the case when \(\xi\) follows the normal distribution generally leads to smaller errors than the Laplace one. This observation is further supported by the second row of Figures~\ref{fig:laplace} and~\ref{fig:norm}, which depicts the true densities of \(\xi\) superimposed with their estimates obtained for 100 samples of size \(n=1000\), as the estimates for the normal law appear to be much more stable than those for the Laplace distribution. It is worth mentioning that these results are fully coherent with our theoretical findings, yielding the faster rate of convergence in the  case of the class \(\mathcal{E}(\gamma,M)\). All in all, we conclude that the proposed estimation method allows to obtain favourable results for the considered examples, and hence can successfully be employed for the problems of this kind.

Now we would like to analyse the behaviour of the adaptive estimation procedure for the truncation parameter \(U_n\). To this end, let us consider the case when \(N\) has a shifted Poisson distribution, and \(\xi_1\) follows the standard Laplace law. Recall that \(p_\xi \in \mathcal{C}(1,1),\) and therefore we can take \(\bar{\beta}=1.\) In this example, the intensity parameter of \(N\) is fixed as \(\lambda=0.1\). Since in this case \(\mathtt{c}_\lambda\approx 9.508>1\), as was discussed in Example~\ref{exm:poisson}, the condition~\eqref{eq:ex-ass-h} holds with some \(1<\rho_0<\mathtt{c}_\lambda\) and \(\varkappa\) defined by~\eqref{varkappa0}. Taking \(\rho_0=5\), we get that \(\varkappa\approx 2.218\). Let us emphasise that Theorem~\ref{thm25} is valid for any \(h,\) and for this study we take \(h=1\).  Finally, since \(\E[N]\approx 1.051\), we conclude that~\eqref{adaptive_bound} holds with \(\zeta=2/3\) for any \(\ell>\kappa^2 M\E[N]h\approx 5.17\), provided that \(K_n=O(n^{\delta})\) with some \(\delta\in(0,1-\zeta)\). In the numerical example we take \(K_n=\lfloor n^{1/4} \rfloor\).

\begin{figure}[t]
\includegraphics[width=1\linewidth]{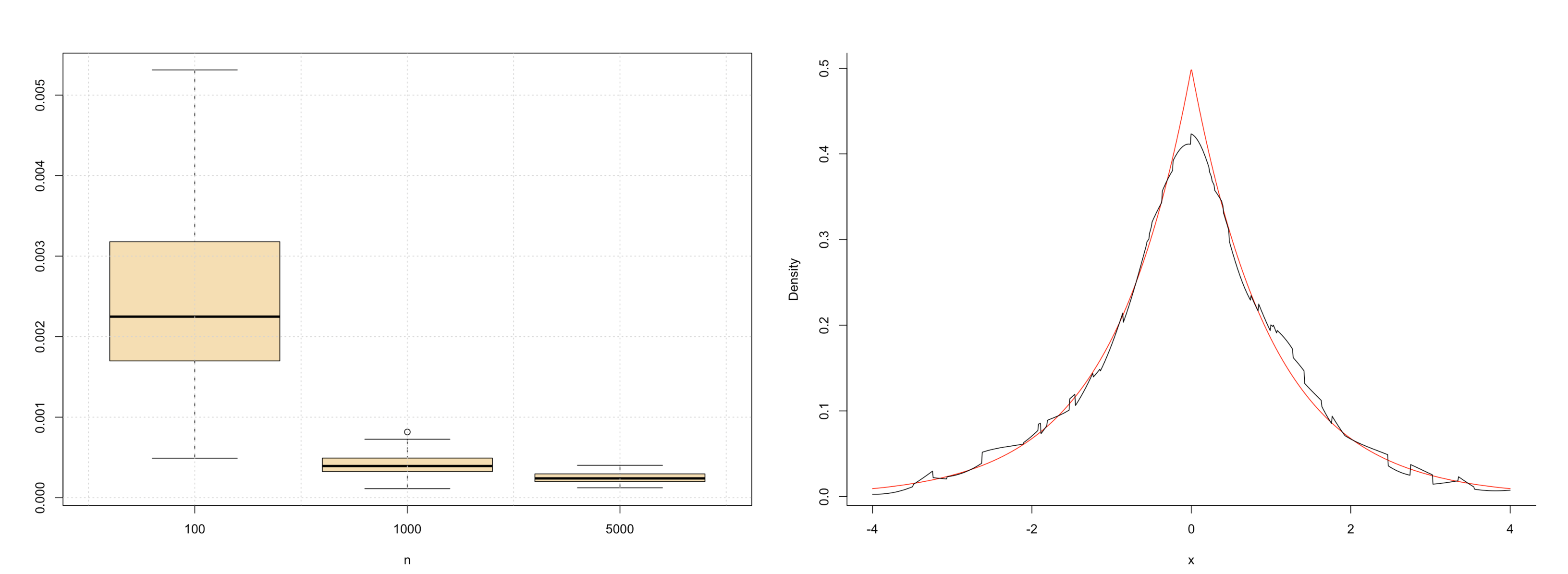}
	\caption{Left: the errors~\eqref{error_def} of the estimator~\eqref{est2} with \(U_n\) chosen by the adaptive procedure. Right: the true (red) density and its estimate (black) based on a sample of size \(n=1000\). }
	\label{fig:adaptive_lap}
\end{figure}

Figure~\ref{fig:adaptive_lap} shows the results of applying the adaptive estimation procedure for the density estimation at the points \(x\), which are taken on an equidistant grid with \(1000\) values from \(-4\) to \(4\). The first plot indicates the errors~\eqref{error_def} of the proposed estimator~\eqref{est2} with \(U_n\) being estimated adaptively based on samples of size \(n=100\), \(1000\) and \(5000\). As before, the errors were computed based on 100 samples of each size. It can be seen that the empirical errors of the estimator decline with the growth of sample size, and have values comparable to those obtained using Theorem~\ref{thm:minmax-pol}. The second plot in Figure~\ref{fig:adaptive_lap} demonstrates the estimated density constructed based on a sample of size \(n=1000\) alongside the true one. The curves are rather close to each other, even though the density estimate appears to be less stable compared to the one with \(U_n\) chosen based on Theorem~\ref{thm:minmax-pol}.

\section{Real data example}
\label{sec:RD}
Now we would like to demonstrate the efficiency of the proposed estimation procedure for the real-life analysis. To this end, let us consider the dataset \texttt{freMTPL}, comprising the data on motor third-part liability policies for different regions of France and available in the \texttt{CASdatasets} package in \texttt{R}. More precisely, we are working with the two sub-datasets: \texttt{FreMTPL2freq}, indicating the number of claims for a given policy number, and \texttt{FreMTPL2sev}, representing the claim amounts. Since for some policy numbers either the number or the amount of claims are not available, we only leave those common for both datasets. In addition, as our estimation method assumes that \(N\) is strictly positive, we only consider the observations for which the number of claims is greater than or equal to one. In this study, we are focusing on the data for the Burgundy region, constituting the total of 345 observations. As the values for the claim amounts are rather large, in what follows we aim at modelling their logarithms.
\begin{figure}[t]
\includegraphics[width=1\linewidth]{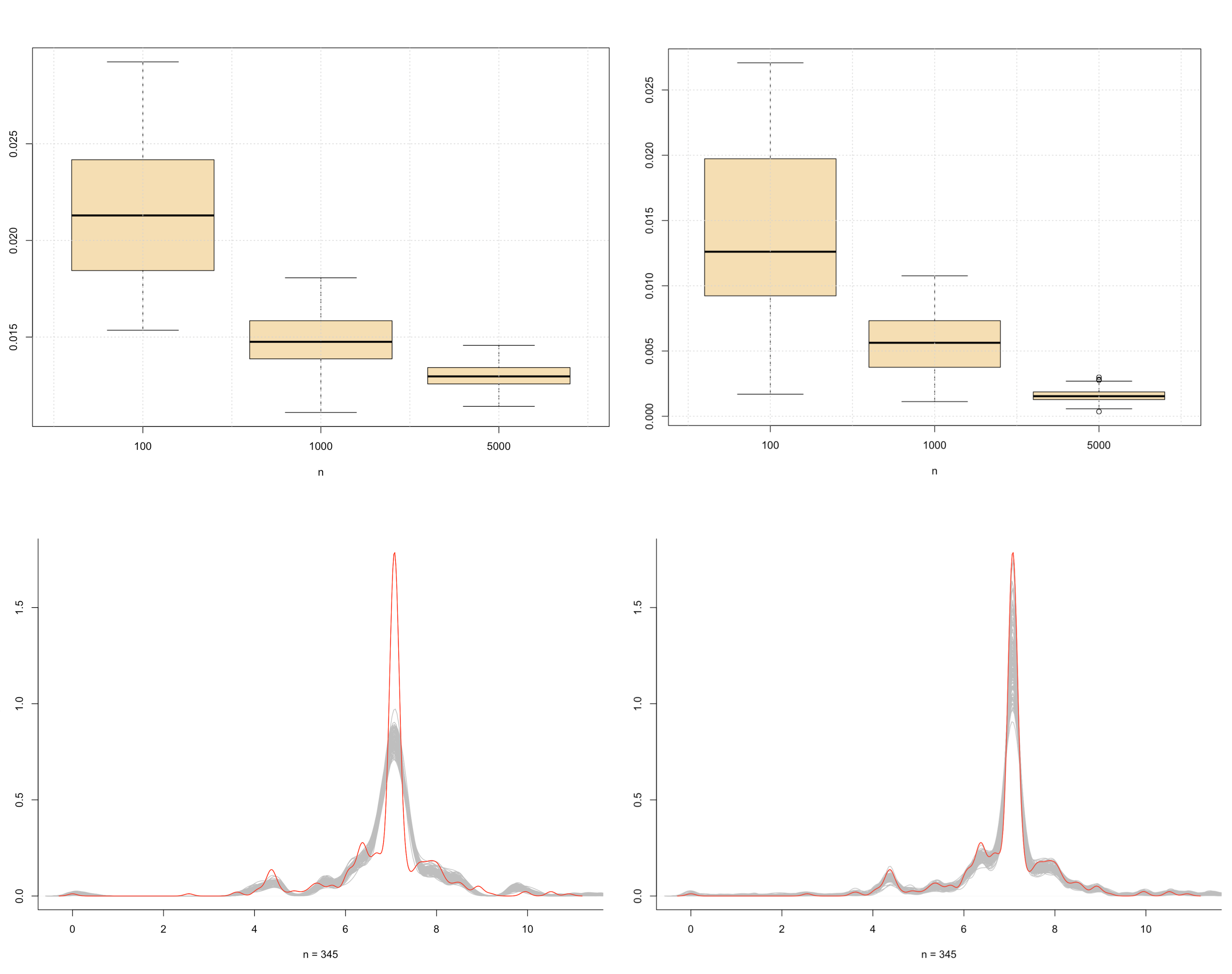}
	\caption{Top: the errors~\eqref{err_RD} obtained for the estimator~\eqref{est2} with \(U_n\) chosen by the adaptive procedure (left) and \(U_n=54\) chosen by the grid search  (right) based on 100 simulation runs. Bottom: the kernel density estimates for the true (red) and simulated (grey) cumulative claim amounts with \(U_n\) estimated adaptively (left) and \(U_n=54\) (right).}
	\label{fig:RD}
\end{figure}

The preliminary step for applying the estimation technique described in Section~\ref{sec:main} is determining the distribution of \(N\). Since for the considered data the maximal number of claims per policy is equal to two, we are fitting the two-point distribution as in Example~\ref{exm:two-points} with parameter \(p=\P\{N=1\}\) being estimated as the proportion of observations for which the number of claims is equal to one. The resulting value is equal to \(\widehat{p}\approx 0.9014\). Hence, in what follows it is assumed that \(N\) has the two-point distribution with this parameter.

Now we turn towards recovering the law of each individual claim \(\xi_1\). The key step in using the estimator~\eqref{est2} is choosing the cutoff parameter \(U_n\). For this one can apply the adaptive estimation procedure presented in Subsection~\ref{sec23}. From Example~\ref{exm:two-points} it is known that the condition~\eqref{eq:ex-ass-h} holds with \(1<\rho_0<\rho^*\), where  \(\rho^*\) can be estimated by \(\widehat{\rho}^*=\widehat{p}^2/(4(1-\widehat{p}))\approx 2.061>1\).  Choosing \(\rho_0=1\), we get \(\varkappa\approx 1.546\), see~\eqref{varkappa0}. Next, assuming that \(p_\xi \in\mathcal{C}(\beta,M)\) with some \(\beta\in(1/2, \bar\beta]\), where \(\bar\beta=2\),  we can estimate the  parameter \(M\)  by \(\widehat{M}\approx 44.584\) using~\eqref{M_est}. As in the simulated example, we take \(h=1\). Since \(\E[N]\approx 1.099\), we get by Theorem~\ref{thm25} that~\eqref{adaptive_bound} holds for any \(\ell>117.056\) with \(\zeta=2\bar{\beta}/(2\bar{\beta}+1)\in[1/2,1)\), provided that \(K_n=O(n^\delta)\) for \(\delta\in(0,1-\zeta)\). Our empirical study shows that the best performance is achieved with the choice \(K_n=2 \lfloor n^{1/4} \rfloor\). 

The left-hand side of Figure~\ref{fig:RD} demonstrates the results obtained for the proposed estimator with \(U_n\) chosen by the adaptive procedure. The top left plot indicates the boxplots of errors~\eqref{err_RD} of the proposed estimator for this cutoff parameter obtained for 100 simulation runs across various sample sizes. It can be seen that the values of errors are rather small, and that they decrease as the number of observations grows. The bottom left plot represents the kernel density estimates of the true and 100 simulated cumulative logarithms of claim amounts. In the latter case, the sample size is taken equal to \(n=1000\). It can be seen that the density curves are very close to each other.

Since the adaptive procedure for the choice of \(U_n\) requires the computation of~\eqref{k_hat} for every point at which the density is evaluated, it might be time-consuming if both the number of such points and the sample size are large. In that case, one can employ the following procedure. Let the values of \(U_n\) be taken on a grid from \(50\) to \(100\) with a step equal to \(1\) and calculate the estimator \(\widehat{p}_{\xi}^{U_n}\) defined by~\eqref{est2} for each \(U_n\). Then simulate \(25\) samples of size \(n=1000\) from~\eqref{randomsum_main} with \(\xi_1\) having the density \(\widehat{p}_{\xi}^{U_n}\) and \(N\) having the two-point law, and for each of these samples calculate the error 
\begin{equation}\label{err_RD}
err_{U_n,k}:=\frac{1}{J}\sum\limits_{j=1} ^J \left(p_X(x_j)-\widehat{p}_X^{k,U_n}(x_j)\right)^2,
\end{equation}
where \(p_X\) is the kernel density estimate for the true cumulative (logarithmic) claim amounts, \(\widehat{p}_X^{k,U_n}\) is the kernel density estimate obtained for a \(k\)-th simulated sample from~\eqref{randomsum_main}, \(1\leq k\leq 25\), with \(\xi_1\) having the density \(\widehat{p}_{\xi}^{U_n}\), and \(\{x_j\}_{1\leq j\leq J}\) are the values of an equidistant grid from \(x_1=-0.296\) to \(x_J=11.2\), \(J=1000\). The value of \(U_n\) corresponding to the smallest average error \(\frac{1}{25}\sum_{k=1} ^{25} err_{U_n,k}\) can be considered optimal. 

In our case, the procedure above leads to the value \(U_n=54\). The right-hand side of Figure~\ref{fig:RD} demonstrates the results obtained for the estimator~\eqref{est2} with this truncation parameter. Similar to before, the boxplots of errors~\eqref{err_RD} indicate that the estimator becomes more accurate as the sample size grows, and the values of errors are a bit smaller than those obtained in the previous case. The densities of simulated logarithms of claim amounts presented in the bottom right plot, again constructed based on 100 samples of size \(n=1000\), are very close to the density of the true ones, and it can be seen that the estimates reflect all the key characteristics present in the density of the original data.

We conclude that both estimation methods (adaptive approach and grid search) allow to successfully recover the density of individual claim payments, and can therefore be efficiently applied for modelling the claim amounts in insurance. Not surprisingly, the results of the grid search procedure are slightly better than those of the adaptive approach.

\section{The case of deterministic \(N\)}\label{sec4}


A natural question, which is closely related to the considered problem, is how to infer the distribution
of a random variable \(\xi\) from the distribution of the sum of its independent copies \(\xi_1+\xi_2+\ldots+\xi_m\) for a specified \(m\in \mathbb{N}\). This task essentially boils down to the intricate process of reconstructing the characteristic function \(\phi_\xi\) from its powers, which is recognized as an inherently difficult problem due to its ill-posed nature, see Gorenflo and Hofmann (\citeyear{GH1994}). An illustrative example can demonstrate this complexity: it has been established that there exists a distribution function \(F\) such that the distribution of the sum of any number of independent random variables adhering to the law \(F\) does not uniquely determine \(F\). To elucidate this point, one might consider a distribution \(F\) defined by the density function:
\[
F'(x)=p_{\xi}(x)=\frac{1-\cos(x)}{\pi x^2}(1-\cos(2x)), \qquad 
x \in \R.
\]
Then by defining for any \(m,\)  
\begin{eqnarray*}
G'_m(x)=\frac{1-\cos(x)}{\pi x^2}\left(1-\cos\left(2x+\frac{2\pi}{m}\right)\right)
\end{eqnarray*}
it is easy to show that 
\begin{eqnarray*}
\underbrace{F\star\ldots\star F}_{m}=\underbrace{G_m\star\ldots\star G_m}_{m}.
\end{eqnarray*}
This example underscores the nuanced challenges encountered in this problem, highlighting the need for careful consideration and more restrictive assumptions.
In this respect, the following result can be proved. Let us assume that the characteristic function \(\phi_\xi\) doesn't vanish on \(\mathbb{R}\), and therefore \((\phi_X)^{1/m}\) is well defined. Then an estimate  of $p_\xi$ can be defined as
$$
\widehat p_{\xi}(x)=\frac{1}{2\pi} \int_{-U_n}^{U_n} e^{-\i\omega x}(\widehat\phi_X(\omega))^{1/m}\,d\omega.
$$
Since $(\widehat\phi_X)^{1/m}$ is well defined only on the interval where $\widehat\phi_X\neq 0$  we will 
consider  the event 
\[
\B_n(\varkappa):=\left\{\sup_{\omega\in [-U_n,U_n]}\left|\frac{\widehat\phi_X(\omega)-\phi_X(\omega)}{\phi_X(\omega)}\right|\leq \varkappa\right\}
\]  
for some  \(\varkappa >  0\). As before, we will use the notation \(\B_n:=\B_n(\varkappa)\) for shortness. The following result holds.
\begin{theorem}\label{thm:gen-bound-power}
Suppose that \(p_\xi \in \mathcal{C}(\beta,M)\) with 
 some \(\beta>0\) and \(M>0.\) Assume that   \(\inf_{n>n_0} \P(\B_n)>0\) for some \(\kappa>0\). Then it holds for any \(x\in \mathbb{R},\) 
\begin{eqnarray}
\label{thm2bound-power}
\E\left[\left|\widehat{p}_\xi(x)-p_\xi(x)\right|^2\Bigr|\B_n\right] 
 \lesssim U_n^{-2\beta}+\frac{U_{n}^{1+(\beta+1)(2m-1)} }{n},
\end{eqnarray}
provided \(U_{n}^{(\beta+1)m}\,n^{-1/2}=o(1),\) \(n\to \infty.\)
Here \(\lesssim\) stands for inequality up to an absolute constant not depending on \(n.\) \end{theorem}
\begin{corollary}\label{cor42}
By choosing \(U_{n}=n^{1/(\beta+2m(\beta+1))}\),
we derive
\begin{eqnarray*}
\E\left[\left|\widehat{p}_\xi(x)-p_\xi(x)\right|^2\Bigr|\B_n\right]\lesssim 
n^{-2\beta/(\beta+2m(\beta+1))},
\end{eqnarray*}
where \(\P(\B_n^c)\leq  \left(\sqrt{n}U_n\right)^{-2}\).
\end{corollary}

\section{Discussion} \label{sec_con}

In this paper, we study the problem of reconstructing the distribution of \(\xi\) from the observations of the compound models \(\xi_1+..+\xi_N\), where the distribution of \(N\) is assumed to be known.  This task, known as decompounding, was previously  considered mainly for the cases when the inverse function to the probability generating function of \(N\) is well defined, while the general case was examined only in the paper by B{\o}gsted and Pitts, \citeyear{Bogsted}. Our findings significantly improve the method and its theoretical basis, showing that the  rates of convergence achieve minimax optimality in some classes,  see Section~\ref{sec23}.  

In our study we used the assumption \(\P(N=1)>0\). If this assumption is violated, the inference is essentially more difficult, as was mentioned also in previous papers on the topic (see, e.g., discussion section in the paper by B{\o}gsted and Pitts, \citeyear{Bogsted}). 
  The latter case includes the scenario of deterministic \(N\), which we considered in Section~\ref{sec4}.  The direct comparison of Corollary~\ref{cor42} with Theorems~\ref{thm:minmax-pol} and \ref{thm:lb} leads to the conclusion, that the rates of convergence in  the deterministic case are slower than in our setup. The in-depth examination of the case \(\P(N=1)=0\) may be a promising direction for future exploration.
 
Note that the inverse problem of estimating \(N\)  from the observations of \(\xi_1+..+\xi_N\), provided that the distribution of \(\xi\) is known, was considered in our previous paper (Belomestny, Morozova, Panov, \citeyear{BPM2024}). The method presented there is based on a completely different technique, using the superposition of Mellin and Laplace transforms. Nevertheless, the study of the case the distribution of both \(\xi\) and \(N\) are not known, including the identifiability conditons,  is lacking in the existing literature. \newline

\section{Proofs}\label{sec5}

\subsection{Proof of Theorem~\ref{thm:gen-bound}}
It holds for all \(x\in\R\),
\begin{multline*}
\E\left[\widehat{p}_\xi(x)|\A_n\right]-p_\xi(x)\\ 
= \frac{1}{2\pi} \int_{|u|\leq U_n} e^{-\i ux}\, \E\left[e^{-\L^{-1}_N\left(\widehat{\phi}_X(u)\right)}-e^{-\L^{-1}_N\left(\phi_X(u)\right)}\big|\A_n\right]\,du 
\\
 - \frac{1}{2\pi}\int_{|u|>U_n} e^{-\i ux + \psi_\xi(u)}\,du
=  I_1 + I_2.
\end{multline*}
Applying Lemma~\ref{lem:taylor} to the function
\(\H(z)=\exp(-\L^{-1}_N(z))\) with \(k=2\), \(z=\widehat{\phi}_X(u)\), \(a=\phi_X(u)\),  we get 
\begin{multline*}
\E\Bigl[e^{-\L^{-1}_N(\widehat{\phi}_X(u))}-e^{-\L^{-1}_N(\phi_X(u))}|\A_n\Bigr] \\
=\E\Bigl[\E_{\tau}[g_{\tau}(\widehat{\phi}_X(u),\phi_X(u))](\widehat{\phi}_X(u)-\phi_X(u))^2\Big|\A_n\Bigr],
\end{multline*}
where
\[
g_{\tau}(\widehat{\phi}_X(u),\phi_X(u)) := (1-\tau)\H''\bigl(\phi_{X,\tau}(u)\bigr),\,\, \phi_{X,\tau}(u):=\phi_X(u)+\tau(\widehat{\phi}_X(u)-\phi_X(u)).
\]
On the event \(\A_n\), we have  \(|\H''( \phi_{X,\tau}(u))|\leq \varkappa\) for all \(\tau \in [0,1]\) and \(|u| \leq U_n\), and hence
\begin{eqnarray*}
|I_1|&\leq & \frac{\varkappa}{2\pi} \int_{|u|\leq U_n} \E[(\widehat{\phi}_X(u)-\phi_X(u))^2|\A_n]\,du
\\
&\lesssim &  \frac{\varkappa}{\P(\A_n)} \int_{|u|\leq U_n} \frac{1-|\phi_X(u)|^2}{n}\,du \lesssim \frac{\varkappa U_n}{n q_n}.
\end{eqnarray*}
Furthermore
\begin{eqnarray*}
\Var\left(\widehat{p}_\xi(x)|\A_n\right) &=& \Var\left(\int_{-U_n} ^{U_n} e^{-\i ux-\L^{-1}_N\left(\widehat{\phi}_X(u)\right)}\,du\Bigr|\A_n\right) \\
&=& \Var\left(\int_{-U_n} ^{U_n} e^{-\i ux} \H'\left(\phi_X(u)\right)\left(\widehat{\phi}_X(u)-\phi_X(u)\right)\,du\right. \\
&& \left. \hspace{1cm}+ \int_{-U_n} ^{U_n} e^{-\i ux} \E_{\tau}\left[g_{\tau}\left(\widehat{\phi}_X(u),\phi_X(u)\right)\right] \right.\\
&&\left.\hspace{3cm}\times\left(\widehat{\phi}_X(u)-\phi_X(u)\right)^2 \,du \Bigr|\A_n \right) \\
&=:& S_1 + 2S_2 + S_3,
\end{eqnarray*}
where 
\begin{eqnarray*}
S_1 &=& \Var\left(\int_{-U_n} ^{U_n} e^{-\i ux}\H'(\phi_X(u))\left(\widehat{\phi}_X(u)-\phi_X(u)\right)\,du \Bigr|\A_n\right),\\
S_2 &=& \int_{-U_n} ^{U_n} \int_{-U_n} ^{U_n} e^{-\i (u+v)x}\H'(\phi_X(u)) \cov\Bigl(\widehat{\phi}_X(u)-\phi_X(u), \Bigr.\\
&& \hspace{2cm}\Bigl.\E_{\tau}\left[g_{\tau}\left(\widehat{\phi}_X(v),\phi_X(v)\right)\right]\left(\widehat{\phi}_X(v)-\phi_X(v)\right)^2\Bigr|\A_n\Bigr)\,du\,dv, \\
S_3 &=& \Var\left(\int_{-U_n} ^{U_n} e^{-\i ux} \E_{\tau}\left[g_{\tau}\left(\widehat{\phi}_X(u),\phi_X(u)\right)\right] \left(\widehat{\phi}_X(u)-\phi_X(u)\right)^2 \,du \Bigr|\A_n\right).
\end{eqnarray*}
For \(S_1\) we have
 \begin{multline*}
 S_1 = \int_{-U_n} ^{U_n} \int_{-U_n} ^{U_n} e^{-\i(u-v)x}\H'(\phi_X(u))\overline{\H'(\phi_X(v))}\\
 	\times\frac{
 	\E\left[
 	\left(\widehat{\phi}_X(u)-\phi_X(u)\right)
 	\left(\overline{\widehat{\phi}_X(v)-\phi_X(v)}\right)
 	\mathbb{I}\{\A_n\}\right]
 	}{\P(\A_n)}\,du\,dv,
 \end{multline*}
where 
 \begin{multline*}
 	\E\left[
 	\left(\widehat{\phi}_X(u)-\phi_X(u)\right)
 	\left(\overline{\widehat{\phi}_X(v)-\phi_X(v)}\right)
 	\mathbb{I}\{\A_n\}\right] \\
 	= \cov\left(\widehat{\phi}_X(u),\widehat{\phi}_X(v)\right)-
 	\E\left[
 	\left(\widehat{\phi}_X(u)-\phi_X(u)\right)
 	\left(\overline{\widehat{\phi}_X(v)-\phi_X(v)}\right)
 	\mathbb{I}\{\A_n^c\}\right].
 \end{multline*}
For the first summand in the expression above we have
\[
\cov\left(\widehat{\phi}_X(u),\widehat{\phi}_X(v)\right) 
= \frac{1}{n}\left(\phi_X(u-v)-\phi_X(u)\overline{\phi_X(v)}\right).
\]
For the second summand, we will use the inequality 
\begin{eqnarray}\label{MZB}
\Vert \widehat{\phi}_X(u)-\phi_X(u)\Vert_{L^p}\lesssim n^{-1/2}, \qquad \forall \; u\in[-U_n,U_n], \quad \forall \; p\geq 1,\end{eqnarray}
which follows from the Marcinkiewicz - Zygmund - Burkholder inequality (see Lin and Bai,  \citeyear{LinBai}, Section~9.7),
\begin{eqnarray*}
\Vert \widehat{\phi}_X(u)-\phi_X(u)\Vert_{L^p}
= \biggl(
\E \Bigl|
\frac{1}{n} \sum_{k=1}^n 
\eta_k \Bigr|^p 
\biggr)^{1/p}
\leq \biggl( \frac{c_p}{n^p} \E\Bigl(
\sum_{k=1}^n \eta_k^2
\Bigr)^{p/2} \biggr)^{1/p} \leq 2 c_p^{1/p} n^{-1/2},
\end{eqnarray*}
where the random variables \(\eta_k:=e^{\i u X_k} - \E e^{\i u X_k}, k=1,..n,
\) are bounded by 2, and the constant 
\(c_p\) does not depend on the distribution of \(X\). 
The formula \eqref{MZB} together with the 
 the Cauchy-Schwarz inequality yield \begin{multline*}
	\E\left[
 	\left(\widehat{\phi}_X(u)-\phi_X(u)\right)
 	\left(\overline{\widehat{\phi}_X(v)-\phi_X(v)}\right)
 	\mathbb{I}\{\A_n^c\}\right]\\
 	\leq \left(\E\left[
 	\left|\widehat{\phi}_X(u)-\phi_X(u)\right|^{4}
 	\right]
 	\E\left[
 	\left|\widehat{\phi}_X(v)-\phi_X(v)\right|^{4}
 	\right]
 	\right)^{1/4}
 	(\P(\A_n^c))^{1/2}
	\lesssim 
	\frac{(1-q_n)^{1/2}}{n}.
\end{multline*}
Hence, we further get
\begin{multline*}
	|S_1| \lesssim \frac{1}{nq_n}\int_{-U_n} ^{U_n} \int_{-U_n} ^{U_n} |\H'(\phi_X(u))\H'(\phi_X(v))\phi_X(u-v)|\,du\,dv\\
	 +\frac{(1-q_n)^{1/2}}{n q_n} \int_{-U_n} ^{U_n} \int_{-U_n} ^{U_n} |\H'(\phi_X(u))\H'(\phi_X(v))\phi_X(u)\overline{\phi_X(v)}|\,du\,dv.
\end{multline*}
Using again the Cauchy-Schwarz inequality, we get
\begin{eqnarray*}
|S_1| &\lesssim&  \frac{1}{nq_n}\int_{-U_n} ^{U_n} |\H'(\phi_X(u)|^2\,du \int_{\R} |\phi_X(v)|\,dv 
\\ && \hspace{4.5cm}+  \frac{(1-q_n)^{1/2}}{n q_n} \left(\int_{-U_n} ^{U_n}  |\H'(\phi_X(u))|\,du\right)^2
  \\
  &\lesssim&\frac{U_n \varkappa^2 }{nq_n}\Bigl( C_\phi \E[N]+(1-q_n)^{1/2}U_n \Bigr),
\end{eqnarray*}
where we also applied that \(\int_{\R} |\phi_X(v)|\,dv  \leq \E[N] \int_{\R} |\phi_\xi(v)|\,dv = \E[N] C_\phi.\)
As for \(S_2\), we can establish the following upper bound by the  application of the H\"older  inequality,
\begin{multline*}
|S_2|
\leq 
  \varkappa \int_{-U_n} ^{U_n} \int_{-U_n} ^{U_n}
\bigl\Vert \widehat{\phi}_{X}(u)-\phi_{X}(u)\bigr\Vert _{L^{4}}\bigl\Vert (\widehat{\phi}_{X}(v)-\phi_{X}(v))^{2}\bigr\Vert _{L^{4}} \\ 
\hspace{4cm} \times (\mathsf{P}(\mathcal{A}_{n}))^{-1/2}\,du\,dv
\lesssim  \frac{U^2_n \varkappa}{q_n^{1/2} n^{3/2}}.
\end{multline*}
Analogously,
\begin{multline*}
|S_3|
\leq    \varkappa^2\int_{-U_n} ^{U_n} \int_{-U_n} ^{U_n}
\bigl\Vert (\widehat{\phi}_{X}(u)-\phi_{X}(u))^2\bigr\Vert _{L^{4}}\bigl\Vert (\widehat{\phi}_{X}(v)-\phi_{X}(v))^{2}\bigr\Vert _{L^{4}}\\
\hspace{4cm}\times (\mathsf{P}(\mathcal{A}_{n}))^{-1/2}\,du\,dv
\lesssim  \frac{U^2_n \varkappa^2}{q_n^{1/2} n^2}.
\end{multline*}
\subsection{Proof of Theorem~\ref{thm:minmax-pol}}  

The assumption \(q_n \geq 1-U_n^{-2}\) yields that the fourth and the fifth summands in the right-hand side of~\eqref{thm2bound} are of smaller order than the first three. 

(i) Using the assumption on the behaviour of \(\phi_\xi,\) we get
\begin{eqnarray*}
\biggl(\int_{|u|>U_n} \left|\phi_{\xi}(u)\right|\,du\biggr)^2 
\leq
\frac{4 M^2}{\beta^2}U_n^{-2\beta},\qquad
C_\phi = \int_{\R}|\phi_\xi(u)| du \leq 2 M \left(1+ \frac{1}{\beta} \right),
\end{eqnarray*} 
which leads to the estimate 
\begin{eqnarray*}
\E\left[\left(\widehat{p}_\xi(x)-p_\xi(x)\right)^2\Bigr|\A_n\right] 
\lesssim 
\frac{M^2}{\beta^2} U_n^{-2\beta} + 
\frac{U_n}{n} 
 \varkappa^2 M \left(1 +\frac{1}{\beta}\right) \E[N] 
+
\varkappa^2 \frac{U_n^2}{n^2 q_n^2}.
\end{eqnarray*}
The choice \(U_{n}=n^{1/(1+2\beta)}\) yields that the first two summands are of order \(n^{-2\beta/(1+2\beta)},\) while the last summand converges faster, provided \(\beta>1/2\). 

(ii) Similarly, using the properties of gamma and incomplete gamma functions, we get
\begin{eqnarray*}
\biggl(\int_{|u|>U_n} \left|\phi_{\xi}(u)\right|\,du\biggr)^2 
&\leq&
\frac{4 M^2}{\gamma^2 (c_\gamma)^{2/\gamma}}
U_{n}^{2(1-\gamma)}e^{-2c_\gamma U_{n}^{\gamma}},\\
C_\phi = \int_{\R}|\phi_\xi(u)| du &\leq& \frac{2 M}{\gamma c_\gamma^{1/\gamma}} \Gamma(1/\gamma) \leq 2M c_\gamma^{-1/\gamma} \max\bigl(1, \Gamma(\gamma^{-1}_\circ+1)\bigr),
\end{eqnarray*} 
and
\begin{multline*}
\E\left[\left(\widehat{p}_\xi(x)-p_\xi(x)\right)^2\Bigr|\A_n\right] 
\lesssim 
\frac{ M^2}{\gamma^2 (c_\gamma)^{2/\gamma}}
U_{n}^{2(1-\gamma)}e^{-2c_\gamma U_{n}^{\gamma}}\\+ 
\frac{U_n}{n}   \varkappa^2 M \E[N] 
\frac{\max\bigl(1, \Gamma(\gamma^{-1}_\circ+1)\bigr)}{c_\gamma^{1/\gamma}} 
+
\varkappa^2 \frac{U_n^2}{n^2 q_n^2}.
\end{multline*}
Again, under appropriate choice of the sequence \(U_n\) the first two summands yield the required rate of convergence, while the last summand converges faster.

 \subsection{Proof of Theorem~\ref{thm:lb}}
\begin{itemize}
\item[(i)]
Let us start with the class \(\mathcal{C}(\beta,M)\).
Set
\[
K_{0,\mathcal{C}}(x)=\prod\limits _{k=1}^{\infty}\left(\frac{\sin(a_{k}x)}{a_{k}x}\right)^{2}
\]
with $a_{k}=2^{-k-1},$ $k\in\mathbb{N}.$ It can be observed that $K_{0,\mathcal{C}}$
is a characteristic function of the random variable $Z_1=\sum_{k=1}^{\infty}a_{k}(U_{k}+\widetilde{U}_{k})$
where $U_{k},\widetilde{U}_{k},$ $k\in\mathbb{N}, $ are jointly independent
random variables with the uniform distribution on $[-1,1].$ In turn, the Fourier transform \(\phi_{K_{0,\mathcal{C}}}\) of \(K_{0,\mathcal{C}}\) can be represented as $\phi_{K_{0,\mathcal{C}}}(u) = 2\pi p_{Z_1}(-u),$ where \(p_{Z_1}\) is the density of \(Z_1\). 
Since $|Z_1|\leq\sum_{k=1}^{\infty}2a_{k}=\sum_{k=1}^{\infty}2^{-k}=1$, the function $\phi_{K_{0,\mathcal{C}}}$ vanishes for $|u|>1$.

Let us now define the function
\[
K_{\mathcal{C}}(x)=\frac{1}{\pi}\frac{\sin(2x)}{x}\frac{K_{0,\mathcal{C}}(x)}{K_{0,\mathcal{C}}(0)}
\]
It can be seen that it is well-defined on $\mathbb{R}$, and, since
\[
\int_{\R} e^{\i ux}\frac{\sin(ax)}{x}\,dx=\pi1_{\{|u|\leq a\}}, \quad \forall a\geq 0,\, \forall x,u\in\R,
\]  
its Fourier transform is given by
\begin{eqnarray}\nonumber
\phi_{K_{\mathcal{C}}}(u) & = &\frac{1}{\pi}\frac{1}{K_{0,\mathcal{C}}(0)}\int e^{
\i ux}\frac{\sin(2x)}{x}K_{0,\mathcal{C}}(x)\,dx\\
&=&\nonumber \frac{1}{\pi K_{0,\mathcal{C}}(0)} \E\left[\pi 1_{\{|Z+u|\leq 2\}}\right]
= \frac{1}{\pi K_{\mathcal{C}}(0)}\int_{-2} ^2 \pi p_{Z_1}(x-u)\,dx\\
\label{phiK}
 & = &\frac{\int_{-2}^{2}\phi_{K_{0,\mathcal{C}}}(u-x)\,dx}{\int_{-1}^{1}\phi_{K_{0,\mathcal{C}}}(s)\,ds},
\end{eqnarray}

where the last line follows from the equalities
 \[
 \phi_{K_{0,\mathcal{C}}}(u)=2\pi p_{Z_1}(-u)
 \quad \text{and} \quad
 K_{0,\mathcal{C}}(0) = \frac{1}{2\pi}\int_{-1}^{1}\phi_{K_{0,\mathcal{C}}}(s)\,ds.
 \]

Formula~\eqref{phiK} yields that  $\phi_{K_{\mathcal{C}}}(u)=1$ for $u\in[-1,1]$,  $0<\phi_{K_{\mathcal{C}}}(u)<1$
for all $u\in\mathbb{R},$ and $\phi_{K_{\mathcal{C}}}(u)=0$ for $|u|>3.$

Now consider a distribution of $\xi$ which is infinitely
divisible with the following L\'evy triplet:
\[
b_{1,\mathcal{C}}=0,\quad\sigma_{1,\mathcal{C}}=0,\quad\nu_{1,\mathcal{C}}(x)=\frac{1+\beta}{2}\frac{|x|^{-1}}{1+|x|}.
\]
The characteristic exponent of this distribution is given by 
\begin{align*}
\psi_{1,\mathcal{C}}(u) & =\frac{1+\beta}{2}\int(e^{\i ux}-1)\frac{|x|^{-1}}{1+|x|}\,dx\\
 & =(1+\beta)\int_{0}^{\infty}\frac{\cos(ux)-1}{x}\frac{1}{1+x}\,dx.
\end{align*}
It holds, for $u>0,$
\begin{align}\nonumber
\psi_{1,\mathcal{C}}'(u) & =-(1+\beta)\int_{0}^{\infty}\frac{\sin(ux)}{1+x}\,dx\\
\nonumber
 & =\frac{1+\beta}{u}\int_{0}^{\infty}\frac{1}{1+x}\,d\cos(ux)\\
 \nonumber
 & =-\frac{1+\beta}{u}+\frac{1+\beta}{u}\int_{0}^{\infty}\frac{\cos(ux)}{(1+x)^{2}}\,dx\\
 \label{al1}
 & =-\frac{1+\beta}{u}+\frac{2(1+\beta)}{u^{2}}\int_{0}^{\infty}\frac{\sin(ux)}{(1+x)^{3}}\,dx.
\end{align}
Hence, by integrating from $1$ to $s,$ we derive with some \(c_1>0\)
\[
\left|\psi_{1,\mathcal{C}}(s)+(1+\beta)\log(s)\right|\leq c_{1},\quad s>1.
\]
As a result, the corresponding characteristic function $\phi_{1}(u)$
satisfies 
\[
e^{-c_{1}}|u|^{-1-\beta}\leq|\phi_{1,\mathcal{C}}(u)|\leq e^{c_{1}}|u|^{-1-\beta},\quad|u|>1,
\]
while the density $p_{1,\mathcal{C}}$ of $\xi$ satisfies $p_{1,\mathcal{C}}(x)\geq c_{2}/(1+x^{2})$
for some $c_{2}>0$, since 
\begin{align*}
p_{1,\mathcal{C}}(x) & =\frac{1}{2\pi}\int e^{-\i ux+\psi_{1,\mathcal{C}}(u)}\,du.
\end{align*}
Using the fact that 
\begin{align*}
\psi_{1,\mathcal{C}}^{(2)}(u) & =-(1+\beta)\int_{0}^{\infty}\frac{x\cos(ux)}{1+x}\,dx\\
 & =-\frac{(1+\beta)\delta_0(u)}{2}+(1+\beta)\int_{0}^{\infty}\frac{\cos(ux)}{1+x}\,dx\\
 & =-\frac{(1+\beta)\delta_0(u)}{2}-(1+\beta)\mathrm{ci}(u)\cos(u)-(1+\beta)\mathrm{si}(u)\sin(u)
\end{align*}
with 
\[
\mathrm{ci}(z)=-\int_{z}^{\infty}\frac{\cos(t)}{t}\,dt=\gamma+\log(z)-\int_{0}^{z}\frac{1-\cos(t)}{t}\,dt
\]
and 
\[
\mathrm{si}(z)=-\int_{z}^{\infty}\frac{\sin(t)}{t}\,dt,
\]
we derive for $x>0,$
\begin{eqnarray*}
p_{1,\mathcal{C}}(x) &=& \frac{1}{\pi}\int_{0}^{\infty}\cos(ux)e^{\psi_{1,\mathcal{C}}(u)}\,du\\
 &=&-\frac{1}{\pi}\frac{1}{x}\int\psi_{1}'(u)\sin(ux)e^{\psi_{1,\mathcal{C}}(u)}\,du\\
 &=&-\frac{1}{\pi}\frac{1}{x^{2}}\int\psi_{1,\mathcal{C}}^{(2)}(u)\cos(ux)e^{\psi_{1,\mathcal{C}}(u)}\,du\\
 && \hspace{1cm}- \frac{1}{\pi}\frac{1}{x^{2}}\int(\psi_{1,\mathcal{C}}'(u))^{2}\cos(ux)e^{\psi_{1,\mathcal{C}}(u)}\,du\\
 &=&\frac{1+\beta}{2\pi}\frac{1}{x^{2}}+R(x),
\end{eqnarray*}
where $R(x)=o(x^{-2})$ for $x\to\infty$ since
\[
\int_{0}^{\infty}\log(u)\cos(ux)e^{\psi_{1,\mathcal{C}}(u)}\,du=O(\log(x)/x),\quad x\to\infty.
\]
Now set 
\[
\nu_{2,\mathcal{C}}(x)=\nu_{1,\mathcal{C}}(x)+\varepsilon\delta_{h}^{\mathcal{C}}(x),\quad\delta_{h}^{\mathcal{C}}(x)=h^{-1}K_{\mathcal{C}}(x/h).
\]
One can always choose $\varepsilon$ in such a way that $\nu_{2,\mathcal{C}}$
stays positive on $\mathbb{R}$ and thus can be viewed as the L\'evy
measure. Denote 
\[
\psi_{\xi,i}^{\mathcal{C}}(u)=\int(e^{\i ux}-1)\nu_{i,\mathcal{C}}(x)\,dx,\quad i=1,2.
\]
Then it holds that $\psi_{\xi,2}^{\mathcal{C}}(u)=\psi_{\xi,1}^{\mathcal{C}}(u)+\varepsilon\widehat{\delta}_{h}^{\mathcal{C}}(u)$
with 
\begin{align*}
\widehat{\delta}_{h}^{\mathcal{C}}(u) & :=\int(e^{\i ux}-1)\delta_{h}^{\mathcal{C}}(x)\,dx=\phi_{K_{\mathcal{C}}}(hu)-1,
\end{align*}
since \(\int_\R K_{\mathcal{C}}(x)\,dx=\phi_{K_{\mathcal{C}}}(0)=1\).
Note that $\widehat{\delta}_{h}^{\mathcal{C}}(u)\leq0$ for all $u$ and
\[
\widehat{\delta}_{h}^{\mathcal{C}}(u)=0,\,u\in[-1/h,1/h],\quad\widehat{\delta}_{h}^{\mathcal{C}}(u)=-1,\,|u|>3/h.
\]
Denote by $p_{\xi,1}^{\mathcal{C}}$ and $p_{\xi,2}^{\mathcal{C}}$ the densities of infinitely
divisible distributions with characteristic exponents $\psi_{\xi,1}^{\mathcal{C}}$
and $\psi_{\xi,2}^{\mathcal{C}}$, respectively. Furthermore, set $\phi_{X,i}^{\mathcal{C}}(u)=\mathcal{L}_{N}(-\psi_{\xi,i}^{\mathcal{C}}(u)),$ $i=1,2,$
and let $p_{X,i}^{\mathcal{C}}$ be the density corresponding to the characteristic function $\phi_{X,i}^{\mathcal{C}},$$i=1,2.$
We have 
\[
p_{X,1}^{\mathcal{C}}(x)=\sum_{k=1}^{\infty}\mathtt{p}_k\,(p_{\xi,1}^{\mathcal{C}})^{\star k}(x)\geq \mathtt{p}_1 p_{\xi,1}^{\mathcal{C}}(x)\geq\mathtt{p}_1/(1+x^{2}).
\]
Hence 
\begin{multline*}
\chi^{2}\left(p_{X,1}^{\mathcal{C}},p_{X,2}^{\mathcal{C}}\right)  =  \int_{\mathbb{R}}\frac{\left(p_{X,1}^{\mathcal{C}}(x)-p_{X,2}^{\mathcal{C}}(x)\right)^{2}}{p_{X,1}^{\mathcal{C}}(x)}\,dx\\
 \hspace{-1.5cm} \lesssim  \mathtt{p}_1^{-1}\int_{\mathbb{R}}(1+|x|^{2})\left(p_{X,1}^{\mathcal{C}}(x)-p_{X,2}^{\mathcal{C}}(x)\right)^{2}dx\\
\hspace{-2.5cm}=\frac{1}{2\pi\mathtt{p}_1}\int_{\mathbb{R}_{+}}\left|\phi_{X,1}^{\mathcal{C}}(u)-\phi_{X,2}^{\mathcal{C}}(u)\right|^{2}du\\
  \hspace{0.5cm} +\frac{1}{2\pi\mathtt{p}_1}\int_{\mathbb{R}_{+}}\left\vert (\phi_{X,1}^{\mathcal{C}})^{(1)}(u)-(\phi_{X,2}^{\mathcal{C}})^{(1)}(u)\right\vert ^{2}du\\
 \hspace{-0.5cm}=\frac{1}{2\pi\mathtt{p}_1}\int_{\mathbb{R}_{+}}\Bigl(\mathcal{L}_{N}(-\psi_{\xi,1}^{\mathcal{C}}(u))-\mathcal{L}_{N}(-\psi_{\xi,2}^{\mathcal{C}}(u))\Bigr)^{2}du\\
  +\frac{1}{2\pi\mathtt{p}_1}\int_{\mathbb{R}_{+}}\left(\frac{d}{du}\left[\mathcal{L}_{N}(-\psi_{\xi,1}^{\mathcal{C}}(u))-\mathcal{L}_{N}(-\psi_{\xi,2}^{\mathcal{C}}(u))\right]\right)^{2}du,
\end{multline*}
where the second equality follows from the Parseval-Plancherel identity.
Using the fact that $\widehat{\delta}_{h}^{\mathcal{C}}(u)=0,\,u\in[-1/h,1/h]$,
we get 
\[
\mathcal{L}_{N}(-\psi_{\xi,1}^{\mathcal{C}}(u))-\mathcal{L}_{N}(-\psi_{\xi,2}^{\mathcal{C}}(u))=0,\quad|u|\leq1/h
\]
and, by the mean-value theorem,
\[
\mathcal{L}_{N}(-\psi_{\xi,1}^{\mathcal{C}}(u))-\mathcal{L}_{N}(-\psi_{\xi,2}^{\mathcal{C}}(u))=\mathcal{L}_{N}^{(1)}(-\psi_{\xi,1}^{\mathcal{C}}(u)-\theta\varepsilon\widehat{\delta}_{h}^{\mathcal{C}}(u))\,\varepsilon\widehat{\delta}_{h}^{\mathcal{C}}(u)
\]
for $|u|>1/h$ and some $\theta\in(0,1).$ Noting that
\[
\mathcal{L}_{N}^{(n)}(z)=\sum_{k=1}^{\infty}(-k)^{n}\mathtt{p}_ke^{-kz}
\]
and 
\begin{multline}\label{LNn}
\mathcal{L}_{N}^{(n)}(-\psi_{\xi,1}^{\mathcal{C}}(u)-\theta\varepsilon\widehat{\delta}_{h}^{\mathcal{C}}(u)) \\ =e^{\psi_{\xi,1}^{\mathcal{C}}(u)}\sum_{k=1}^{\infty}(-k)^{n}\mathtt{p}_k e^{(k-1)\psi_{\xi,1}^{\mathcal{C}}(u)+k\theta\varepsilon\widehat{\delta}_{h}^{\mathcal{C}}(u)},
\end{multline}
we get that
\[
\left|\mathcal{L}_{N}(-\psi_{\xi,1}^{\mathcal{C}}(u))-\mathcal{L}_{N}(-\psi_{\xi,2}^{\mathcal{C}}(u))\right|\leq e^{\psi_{\xi,1}^{\mathcal{C}}(u)}\Bigl( \sum_{k=1}^{\infty}k \mathtt{p}_k\Bigr).
\]
Furthermore, 
\begin{multline*}
\frac{d}{du}\left[\mathcal{L}_{N}(-\psi_{\xi,1}^{\mathcal{C}}(u))-\mathcal{L}_{N}(-\psi_{\xi,2}^{\mathcal{C}}(u))\right]=-(\psi_{\xi,1}^{\mathcal{C}})'(u)\mathcal{L}_{N}^{(1)}(-\psi_{\xi,1}^{\mathcal{C}}(u))\\
  \hspace{6cm}+ (\psi_{\xi,2}^{\mathcal{C}})'(u)\mathcal{L}_{N}^{(1)}(-\psi_{\xi,2}^{\mathcal{C}}(u))\\
\hspace{3cm}= -(\psi_{\xi,1}^{\mathcal{C}})'(u)[\mathcal{L}_{N}^{(1)}(-\psi_{\xi,1}^{\mathcal{C}}(u))-\mathcal{L}_{N}^{(1)}(-\psi_{\xi,2}^{\mathcal{C}}(u))]\\
 \hspace{6.2cm}+\varepsilon(\widehat{\delta}_{h}^{\mathcal{C}})'(u)\mathcal{L}_{N}^{(1)}(-\psi_{\xi,2}^{\mathcal{C}}(u))\\
 \hspace{3cm}=-\varepsilon\widehat{\delta}_{h}^{\mathcal{C}}(u) (\psi_{\xi,1}^{\mathcal{C}})'(u)[\mathcal{L}_{N}^{(2)}(-\psi_{\xi,1}^{\mathcal{C}}(u)-\widetilde{\theta}\varepsilon\widehat{\delta}_{h}^{\mathcal{C}}(u))]\\
 \hspace{1cm}+\varepsilon(\widehat{\delta}_{h}^{\mathcal{C}})'(u)\mathcal{L}_{N}^{(1)}(-\psi_{\xi,2}^{\mathcal{C}}(u)),
\end{multline*}
where \(\widetilde{\theta} \in (0,1).\) By analogue to~\eqref{LNn}, we have 
\begin{eqnarray*}
\bigl| \mathcal{L}_{N}^{(2)}(-\psi_{\xi,1}^{\mathcal{C}}(u)-\widetilde{\theta}\varepsilon\widehat{\delta}_{h}^{\mathcal{C}}(u)) \bigr| \leq e^{\psi_{\xi,1}^{\mathcal{C}}(u)}\Bigl(\sum_{k=1}^{\infty}k^{2}\mathtt{p}_k\Bigr).
\end{eqnarray*}
Note that, for all \(u\in\R\),
\(
\bigl| \widetilde{\delta}_{h}^{\mathcal{C}} (u) \bigr| \leq 1
\)
and

\[
\bigl| (\widehat{\delta}_{h}^{\mathcal{C}})'(u) \bigr| =h \bigl| \phi_{K_{\mathcal{C}}}'(hu)\bigr|
=h^2 \frac{|p_{Z_1}(-hu+2) -p_{Z_1}(-hu -2 )|}{\P\bigl\{|Z_1| \leq 1\bigr\}} \lesssim h^2.
\]
In addition, due to~\eqref{al1}, 
\((\psi_{\xi,1}^{\mathcal{C}})'(u)=O(1/u),\;  u \to \infty.\)
As a result,

\begin{eqnarray*}
\chi^{2}\left(p_{X,1}^{\mathcal{C}},p_{X,2}^{\mathcal{C}}\right) & \lesssim &\mathtt{p}_1^{-1}(\E[N])^2\int_{u>1/h}e^{2\psi_{\xi,1}^{\mathcal{C}}(u)}du\\
 & & \hspace{0.5cm} +\mathtt{p}_1^{-1} (\E[N])^2 h^{4}\int_{u>1/h}e^{2\psi_{\xi,1}^{\mathcal{C}}(u)}du\\
 & &\hspace{0.5cm} +\mathtt{p}_1^{-1} \E[N^2] \int_{u>1/h}|(\psi_{\xi,1}^{\mathcal{C}})'(u)|^{2}e^{2\psi_{\xi,1}^{\mathcal{C}}(u)}du\\
 & \lesssim &\mathtt{p}_1^{-1} (\E[N])^2 h^{2\beta+1}.
\end{eqnarray*}
Moreover
\begin{multline*}
\sup\limits_{x\in\R} \left|p_{\xi,1}^{\mathcal{C}}(x)-p_{\xi,2}^{\mathcal{C}}(x)\right|
\geq
p_{\xi,1}(0)^{\mathcal{C}}-p_{\xi,2}^{\mathcal{C}}(0)\\
\geq\frac{1}{\pi}\int_{u>3/h}\phi_{\xi,1}^{\mathcal{C}}(u)(1-e^{-\varepsilon})\,du\gtrsim h^{\beta}.
\end{multline*}
Using the technique for proving the lower bounds based on the chi-squared distance (see, e.g., Chapter 2 in Tsybakov~(\citeyear{Tsybakov}), notably Theorem 2.2), one obtains 
\[
\lim\inf_{n\to\infty}\inf_{\widehat{p}_{\xi}}\sup_{p_{\xi}\in\mathcal{C}(\beta,M)\cap \mathrm{Sym}(\R)}\mathrm{P}\left(\sup_{x\in \R}|\widehat{p}_{\xi}(x)-p_{\xi}(x)|>n^{-\beta/(2\beta+1)}\right)>0.
\]
\item[(ii)] 
The lower bound for the class \(\mathcal{E}(\gamma,M)\) essentially follows from the same ideas as in the proof above. Namely, let us define the kernel
\[
K_{0,\mathcal{E}} (x) = \prod\limits_{k=1} ^{\infty} \left(\frac{\sin(a_k \epsilon x)}{a_k \epsilon x}\right)^2
\]
with some \(\epsilon>0\) and \(a_k=2^{-k-1}\), \(k\in\N\). Again, it can be seen that \(K_{0,\mathcal{E}}\) is a characteristic function of the random variable \(Z_2:=\sum_{k=1} ^{\infty} a_k (U_k^{\epsilon}+\widetilde{U}_k ^{\epsilon})\), where \(U_k^{\epsilon}, \widetilde{U}_k ^{\epsilon}\) are mutually independent and having a uniform on \([-\epsilon,\epsilon]\) distribution. As before, we have that for any \(u\in\R\) the Fourier transform \(\phi_{K_{0,\mathcal{E}}}\) of \(K_{0,\mathcal{E}}\) is equal to \(\phi_{K_{0,\mathcal{E}}}(u)=2\pi p_{Z_2}(-u)\) with \(p_{Z_2}\) being the p.d.f.\ of \(Z_2\), and, since \(|Z_2|\leq \epsilon\), the function \(\phi_{K_{0,\mathcal{E}}}\) vanishes for \(|u|>\epsilon\). Next, let us define the function
\[
K_\mathcal{E} (x) = \frac{1}{\pi} \frac{\sin(x)}{x} \frac{K_{0,\mathcal{E}}(x)}{K_{0,\mathcal{E}} (0)}.
\]
Similar calculations as in~\eqref{phiK} yield that for all \(u\in\R\) its Fourier transform is given by
\[
\phi_{K_\mathcal{E}}(u) = \frac{\int_{-1} ^1 \phi_{K_{0,\mathcal{E}}}(u-x)\,dx}{\int_{-\epsilon} ^{\epsilon} \phi_{K_{0,\mathcal{E}}}(s)\,ds},
\]
implying that \(0<\phi_{K_\mathcal{E}}(u)<1\) for all \(u\in\R\) with \(\phi_{K_\mathcal{E}}(u)=1\) for \(|u|\leq 1-\epsilon\) and \(\phi_{K_\mathcal{E}}(u)=0\) whenever \(|u|>1+\epsilon\).

Now, let us consider the case when \(\xi\) has the Cauchy distribution with the scale parameter \(\gamma>0\) and location parameter equal to zero. Its L\'evy triplet is given by
\[
b_{1,\mathcal{E}}=0,\,\, \sigma_{1,\mathcal{E}}=0,\,\, 
\nu_{1,\mathcal{E}}=c_{1,\gamma} \frac{1}{x^{1+\gamma}}1_{\{(0,\infty)\}}(x) + c_{2,\gamma} \frac{1}{|x|^{1+\gamma}} 1_{\{(-\infty,0)\}}(x),
\]
with \(c_{1,\gamma}\) and \(c_{2,\gamma}\) being some fixed constants, and the characteristic exponent and characteristic function are given for any \(u\in\R\) by \(\psi_{1,\mathcal{E}}(u)=-\gamma |u|\) and \(\phi_{1,\mathcal{E}}(u)=e^{-\gamma|u|}\), respectively. Define also \(\nu_{2,\mathcal{E}}(x)=\nu_{1,\mathcal{E}}(x)+\eps \delta_h^{\mathcal{E}}(x)\) with \(\delta_h ^{\mathcal{E}}=h^{-1} K_{\mathcal{E}}(x/h)\) and \(\eps>0\) small enough so that \(\nu_{2,\mathcal{E}}(x)\geq 0\) for all \(x\in\R\). As before, let us denote by
\[
\psi_{\xi,i}^{\mathcal{E}}(u)=\int_\R (e^{\i ux}-1)\,\nu_{i,\mathcal{E}} (x)\,dx, \quad i=1,2,
\]
the characteristic exponents corresponding to the measures \(\nu_{i,\mathcal{E}}\), and by \(p_{\xi,i} ^{\mathcal{E}}\) the respective densities, \(i=1,2\). In addition, set \(\phi_{X,i}^{\mathcal{E}}(u)=\mathcal{L}_{N}(-\psi_{\xi,i}^{\mathcal{E}}(u)),\) \(u\in\R\), and denote by \(p_{X,i}^{\mathcal{E}}\) the density corresponding to the characteristic function \(\phi_{X,i}^{\mathcal{E}}\), \(i=1,2\).
Since
\[
p_{X,1} ^{\mathcal{E}}(x) \geq \mathtt{p}_1 p_{\xi,1} ^{\mathcal{E}}(x) \geq \mathtt{p}_1 \gamma/(\pi(\gamma^2+x^2)),
\]
we get
\begin{multline*}
\chi^2 (p_{X,1} ^{\mathcal{E}}, p_{X,2} ^{\mathcal{E}}) = \int_\R \frac{\left(p_{X,1} ^{\mathcal{E}}(x)-p_{X,2} ^{\mathcal{E}}(x)\right)^2}{p_{X,1} ^{\mathcal{E}}(x)}\,dx \\
\lesssim  \frac{\gamma}{\mathtt{p}_1} \int_{\R_+} |\phi_{X,1} ^{\mathcal{E}}(u)-\phi_{X,2} ^{\mathcal{E}}(u)|^2\,du \\
\hspace{2.5cm} +\frac{1}{\gamma\mathtt{p}_1} \int_{\R_+} \left|(\phi_{X,1} ^{\mathcal{E}})^{(1)}(u)-(\phi_{X,2}^{\mathcal{E}})^{(1)}(u)\right|^2\,du \\
\hspace{2cm}= \frac{\gamma}{\mathtt{p}_1} \int_{\R_+} \left(\L_N(-\psi_{\xi,1} ^{\mathcal{E}}(u))-\L_N(-\psi_{\xi,2} ^{\mathcal{E}} (u))\right)^2\,du \\
\hspace{-1cm} +\frac{1}{\gamma\mathtt{p}_1} \int_{\R_+} \left(\frac{d}{du}[\L_N(-\psi_{\xi,1} ^{\mathcal{E}}(u))-\L_N(-\psi_{\xi,2} ^{\mathcal{E}} (u))]\right)^2\,du.
\end{multline*}
Observing that \(\psi_{\xi,2}^{\mathcal{E}}(u)=\psi_{\xi,1}^{\mathcal{E}}(u)+\widetilde{\delta}_h^{\mathcal{E}}(u)\), \(u\in\R\), with 
\[
\widetilde{\delta}_h^{\mathcal{E}}(u):=\int_\R (e^{\i ux}-1)\delta_h^{\mathcal{E}}(x)\,dx=\phi_{K_{\mathcal{E}}}(hu)-1
=
\begin{cases}
	0,& |u|\leq (1-\epsilon)/h,\\
	1,& |u|>(1+\epsilon)/h,
\end{cases}
\] 
and \(\widetilde{\delta}_h^{\mathcal{E}}(u)\in(0,1)\) for all \(u\in\R\), by the similar argument as before, we further have that
\begin{eqnarray*}
\chi^2(p_{X,1}^{\mathcal{E}}	,p_{X,2}^\mathcal{E})
&\lesssim & \frac{\gamma}{\mathtt{p}_1}(\E[N])^2 \int_{u>(1-\epsilon)/h} e^{2\psi_{\xi,1}^{\mathcal{E}}(u)}\,du \\
&& \hspace{0.5cm}+ \frac{h^4}{\gamma \mathtt{p}_1}(\E[N])^2 \int_{u>(1-\epsilon)/h} e^{2\psi_{\xi,1}^{\mathcal{E}}(u)}\,du\\
&& \hspace{0.5cm} + \frac{\E[N^2]}{\gamma \mathtt{p}_1} \int_{u>(1-\epsilon)/h} |(\psi_{\xi,1} ^{\mathcal{E}})'(u)|^2e^{2\psi_{\xi,1}^{\mathcal{E}}(u)}\,du \\
&\lesssim & \left((\E[N])^2(1+h^4)+\E[N^2]\right)\mathtt{p}_1 ^{-1} e^{-2\gamma(1-\epsilon)/h}.
\end{eqnarray*}
Moreover, it holds that
\begin{multline*}
\sup\limits_{x\in\R}|p_{\xi,1}^{\mathcal{E}}(x)-p_{\xi,2}^{\mathcal{E}}(x)|
\geq p_{\xi,1}^{\mathcal{E}}(0)-p_{\xi,2}^{\mathcal{E}}(0)\\
\geq \frac{1}{\pi} \int_{u>(1+\epsilon)/h} \phi_{\xi,1} ^{\mathcal{E}}(u)(1-e^{-\eps})\,du
\gtrsim \gamma^{-1} e^{-\gamma(1+\epsilon)/h}.
\end{multline*}
Hence, taking \(h=2(1-\epsilon)\gamma/\log n\), one can again apply the technique for proving the lower bounds based on the \(\chi^2\)-distance to get
\[
\lim\inf\limits_{n\to\infty} \inf_{\widehat{p}_{\xi}} \sup\limits_{p_{\xi}\in\mathcal{E}(\gamma,M)\cap \mathrm{Sym}(\R)}\mathbb{P}\left(\sup_{x\in \R}|\widehat{p}_{\xi}(x)-p_{\xi}(x)|>n^{-\frac{1+\epsilon}{2(1-\epsilon)}}\right)>0.
\]
Since the latter holds for any \(\epsilon>0\), we conclude the claim.

\end{itemize}

\subsection{Proof of Theorem~\ref{thm25}}
Obviously, for any \(k, k' =1..K_n,\)
\begin{eqnarray*}
 \bigl(\widehat{p}_{\xi,k'}(x) - \widehat{p}_{\xi,k}(x)\bigr)^2  
&\leq& V\bigl(\max(k,k')\bigr) + A\bigl(\min(k,k'), x \bigr), \qquad x \in \R.
\end{eqnarray*}
Therefore,   for any \(k =1..K_n,\)
\begin{eqnarray*}
\left(\widehat{p}_{\xi,\hat{k}}(x)-p_\xi(x)\right)^2
&\leq& 3
\left(\widehat{p}_{\xi,\hat{k}}(x)-
\widehat{p}_{\xi,\min(\hat{k},k)}(x)\right)^2 \\
&&\hspace{0.3cm} 
+ 
3\left(
\widehat{p}_{\xi,\min(\hat{k},k)}(x)
- 
\widehat{p}_{\xi,k}(x)\right)^2
+ 
3\left(
\widehat{p}_{\xi,k}(x)
- 
p_\xi(x) \right)^2\\
&\leq& 3 \Bigl( V(\hat{k}) + A(k,x) + V(k) + A(\hat{k},x)\Bigr)
 \\
&& \hspace{5cm} + 3\left(
\widehat{p}_{\xi,k}(x)
- 
p_\xi(x) \right)^2\\
&\leq& 6\Bigl( A(k,x) + V(k) \Bigr) + 
3\left(
\widehat{p}_{\xi,k}(x)
- 
p_\xi(x) \right)^2.
\end{eqnarray*}
We arrive at the following upper bound 
\begin{multline}\label{qqq}
\E \Bigl[ \left(\widehat{p}_{\xi,\hat{k}}(x)-p_\xi(x)\right)^2 \Bigr| \A_n\Bigr] \\
\lesssim  
\min_{k=1.. K_n} \Bigl( 
\E \Bigl[ \left(\widehat{p}_{\xi,k}(x)-p_\xi(x)\right)^2 \Bigr| \A_n\Bigr] + 2 \E\Bigl[A(k,x )\Bigr| \A_n\Bigr] + 2 V(k) \Bigr).
\end{multline}
To complete the proof, we will now focus on the study of \(A(k, x). \) Using the same decomposition of the function \(\H(z)\) as in 
 the proof of Theorem~\ref{thm:minmax-pol}(i), we get
\begin{multline*}
\left(\widehat{p}_{\xi,k}(x)-p_{\xi,k}(x)\right)^2 \lesssim 
\int_{-kh}^{kh} | \H'\left(\phi_X(u)\right)|^2
\left|
\widehat{\phi}_X(u)-\phi_X(u)\right|^2
\\ +
\Bigl(
\E_{\tau}\left[g_{\tau}\left(\widehat{\phi}_X(u),\phi_X(u)\right)\right]
\Bigr)^2\left|\widehat{\phi}_X(u)-\phi_X(u)\right|^4 du,
\end{multline*}
and the upper bound of the mean-squared error 
\begin{eqnarray}\label{adf}
\E \Bigl[  
\left(\widehat{p}_{\xi,k}(x)-p_{\xi,k}(x)\right)^2 
 \Bigr| \A_n\Bigr] 
 &\lesssim&  
 \varkappa^2 h^2 \frac{k^2}{n^2 q_n^2} +
C(M, \varkappa,N) h \frac{k}{n}  \lesssim V(k),
\end{eqnarray}
provided \(\ell >  \varkappa^2 M \E[N]
h\) in the polynomial case.
The last inequality in~\eqref{adf} follows from  our choice of \(K_n\), since \(k/n \leq K_n /n \to 0\) as \(n \to \infty.\) 

In the rest of the proof we need also the inequality 
\begin{eqnarray*}
\P \Bigl\{ 
\bigl| \widehat{\phi}_{X}(u)-\phi_{X}(u) \bigr| \leq v 
\Bigr\} 
\geq 1 - 2 \exp\Bigl\{
 \frac{-v^2 n}{4  +(4/3) v}
\Bigr\}, \qquad u,v \in \R,
\end{eqnarray*}
which follows from the Bernstein inequality. Under the choice \(v=n^{-(\delta+\zeta)/2}\) we get that the probability of the event 
\(
\CC_n :=\Bigl\{ \bigl| \widehat{\phi}_{X}(u)-\phi_{X}(u) \bigr| \leq n^{-(\delta+\zeta)/2} \Bigr\}
\)
is larger or equal to \(1 - 2 e^{-
n^{1-(\delta+\zeta)}/5}\) for any  \(u \in \R\).
On the intersection of events \(\A_n\) and \(\CC_n\) we have 
\begin{eqnarray*}
\left(
\widehat{p}_{\xi,k}(x)-p_{\xi,k}(x)\right)^2
\lesssim K_n n^{-(\delta+\zeta)} \lesssim n^{-\zeta},
\end{eqnarray*}
and therefore for any \( k=1..K_n,\)
 \begin{multline*}
\E \Bigl[ \Bigl( \bigl(\widehat{p}_{\xi,k}(x) - p_{\xi,k}(x)\bigr)^2  -n^{-\zeta} \Bigr)_+\Bigr| \A_n\Bigr] \\
\lesssim \E\Bigl[\bigl(\widehat{p}_{\xi,k}(x) - p_{\xi,k}(x)\bigr)^2 \Bigr| \A_n\Bigr] \P\Bigl\{\overline{\CC_n}\Bigr\}
\lesssim V(k) e^{-
n^{\zeta}/5}.
\end{multline*}
Then we have 
\begin{align*}
\E \Bigl[ \max_{k' = (k+1),...,K_n} &\Bigl(\bigl(\widehat{p}_{\xi,k'}(x) - \widehat{p}_{\xi,k}(x)\bigr)^2 - V(k')\Bigr)_+ \Bigr| \A_n\Bigr] \\
&\lesssim 
\E \Bigl[ \max_{ (k+1),...,K_n}  \Bigl( \bigl(\widehat{p}_{\xi,k'}(x) - p_{\xi,k'}(x)\bigr)^2  - n^{-\zeta} \Bigr)_+\Bigr| \A_n\Bigr] \\
& \qquad\qquad+
\E \Bigl[ \bigl(\widehat{p}_{\xi,k}(x) - p_{\xi,k}(x)\bigr)^2  \Bigr| \A_n\Bigr] \\
&\qquad\qquad + \max_{ (k+1),...,K_n} 
\bigl( p_{\xi,k'}(x) - p_{\xi,k}(x)\bigr)^2 +n^{-\zeta} \\
&\lesssim 
V(k) + 
 \Bigl( \int_{|u| >kh} |\phi_\xi(u)| du \Bigr)^2 + n^{-\zeta}.
\end{align*}
Substituting this inequality into~\eqref{qqq}, we arrive at the desired result.

\subsection{Proof of Theorem~\ref{thm:gen-bound-power}}
 It holds for all \(x\in\R\),
\begin{multline*}
\E\left[\widehat{p}_\xi(x)|\B_n\right]-p_\xi(x) 
= \frac{1}{2\pi} \int_{|u|\leq U_n} e^{-\i ux}\, \E\left[(\widehat{\phi}_X(u))^{1/m}-(\phi_X(u))^{1/m}|\B_n\right]\,du 
\\
 - \frac{1}{2\pi}\int_{|u|>U_n} e^{-\i ux } \phi_\xi(u)\,du
=  I_1 + I_2.
\end{multline*}
Applying Lemma~\ref{lem:taylor} to \(f(z):=z^{1/m}\) with \(k=2\), \(z=\widehat{\phi}_X(u)\), \(a=\phi_X(u)\),  we get 
\begin{multline*}
\E\left[(\widehat{\phi}_X(u))^{1/m}-(\phi_X(u))^{1/m}|\B_n\right]\\ 
=\E\Bigl[\E_{\tau}[g_{\tau}(\widehat{\phi}_X(u),\phi_X(u))](\widehat{\phi}_X(u)-\phi_X(u))^2\Big|\B_n\Bigr]
\end{multline*}
where
\[
g_{\tau}(\widehat{\phi}_X(u),\phi_X(u)) := \frac{1-m}{m^2}(1-\tau)\bigl(\phi_{X,\tau}(u)\bigr)^{1/m-2}.
\]
Note that \(|\phi_{X,\tau}(u)|>(1-\varkappa)|\phi_X(u)|\) on \(\B_n\) and  
\[
|g_{\tau}(\widehat{\phi}_X(u),\phi_X(u))|\leq  \frac{m-1}{m^2} (1-\tau)\bigl((1-\varkappa)|\phi_X(u)|\bigr)^{1/m-2}.
\]
Hence 
\[
|I_1|\lesssim \frac{ \kappa^2(1- \varkappa)^{1/m-2}}{n\,\P(\B_n)} \int_{|u|\leq U_n} |\phi_X(u)|^{1/m} \,du\lesssim \frac{U_n^{-\beta}}{n\,q_n}.
\]
Trivially,  \(|I_2| \leq \int_{|u|>U_n} \left|\phi_{\xi}(u)\right|\,du \lesssim U_n^{-\beta}.\) As for the variance of \(\widehat{p}_\xi,\) we get 
\begin{multline*}
\Var\left(\widehat{p}_\xi(x)|\B_n\right) = \Var\left(\int_{-U_n} ^{U_n} e^{-\i ux}(\widehat{\phi}_X(u))^{1/m}\,du\Bigr|\B_n\right) \\
=\Var\left(\int_{-U_n} ^{U_n} e^{-\i ux} (1/m)(\phi_X(u))^{1/m-1}\left(\widehat{\phi}_X(u)-\phi_X(u)\right)\,du\right. \\
\left. \hspace{1cm}+ \int_{-U_n} ^{U_n} e^{-\i ux} \E_{\tau}\left[g_{\tau}\left(\widehat{\phi}_X(u),\phi_X(u)\right)\right]\left(\widehat{\phi}_X(u)-\phi_X(u)\right)^2 \,du \Bigr|\B_n \right) \\
=: S_1 + 2S_2 + S_3,
\end{multline*}
where 
\begin{eqnarray*}
S_1 &=& \Var\left(\int_{-U_n} ^{U_n} e^{-\i ux} (1/m)(\phi_X(u))^{1/m-1}\left(\widehat{\phi}_X(u)-\phi_X(u)\right)\,du \Bigr|\B_n\right),\\
S_2 &=& \int_{-U_n} ^{U_n} \int_{-U_n} ^{U_n} e^{-\i (u+v)x} (1/m)(\phi_X(u))^{1/m-1} \cov\Bigl(\widehat{\phi}_X(u)-\phi_X(u), \Bigr.\\
&& \hspace{2cm}\Bigl.\E_{\tau}\left[g_{\tau}\left(\widehat{\phi}_X(v),\phi_X(v)\right)\right]\left(\widehat{\phi}_X(v)-\phi_X(v)\right)^2\Bigr|\B_n\Bigr)\,du\,dv, \\
S_3 &=& \Var\left(\int_{-U_n} ^{U_n} e^{-\i ux} \E_{\tau}\left[g_{\tau}\left(\widehat{\phi}_X(u),\phi_X(u)\right)\right] \left(\widehat{\phi}_X(u)-\phi_X(u)\right)^2 \,du \Bigr|\B_n\right).
\end{eqnarray*}
Similarly to the proof of Theorem~2.1, we derive 
\begin{multline*}
	|S_1| \lesssim \frac{1}{n\P(\B_n)}\int_{-U_n} ^{U_n} \int_{-U_n} ^{U_n} |\phi_X(u)|^{1/m-1}|\phi_X(v)|^{1/m-1}|\phi_X(u-v)|\,du\,dv\\
	 +\frac{(\P(\B_n^c))^{1/2}}{n \P(\B_n)} \int_{-U_n} ^{U_n} \int_{-U_n} ^{U_n}|\phi_X(u)|^{1/m-1}|\phi_X(v)|^{1/m-1}\,du\,dv.
\end{multline*}
Using again the Cauchy-Schwarz inequality  we get
\begin{multline*}
|S_1| \lesssim  \frac{1}{n\P(\B_n)}\int_{-U_n} ^{U_n} |\phi_X(u)|^{2(1/m-1)}\,du \int_{\R} |\phi_X(v)|\,dv 
\\
  +  \frac{(\P(\B_n^c))^{1/2}}{n \P(\B_n)} \left(\int_{-U_n} ^{U_n}  |\phi_X(u)|^{1/m-1}\,du\right)^2 \\
   \lesssim \frac{U_n^{1+2(\beta+1) (m-1)}}{n\,q_n}+
 \frac{(1-q_n)^{1/2}}{n\, q_n}   U_n^{2+2(\beta+1)(m-1)}.
 \end{multline*}
As for \(S_2\), we can establish the following upper bound by the  application of the H\"older inequality,
\begin{eqnarray*}
|S_2| &\lesssim &   \frac{1}{(\mathsf{P}(\B_{n}))^{1/2}n^{3/2}} \Bigl( \int_{-U_n} ^{U_n}   |\phi_X(u)|^{1/m-1} \, du \Bigr) \Bigl( \int_{-U_n} ^{U_n}|\phi_X(v)|^{1/m-2}
\,dv \Bigr) \\ 
&&\hspace{6cm} \lesssim \frac{1}{q_n^{1/2}\, n^{3/2}} U_n^{2+(\beta+1)(3m-2)},
\\
|S_3| &\lesssim &  \frac{1}{(\mathsf{P}(\B_{n}))^{1/2}n^{2}} \Bigl( \int_{-U_n} ^{U_n}   |\phi_X(u)|^{1/m-2} \, du \Bigr) \Bigl( \int_{-U_n} ^{U_n}|\phi_X(v)|^{1/m-2}
\,dv \Bigr)\\ 
&&\hspace{6cm} \lesssim \frac{1}{q_n^{1/2}\, n^{2}} U_n^{2+(\beta+1)(4m-2)}.
\end{eqnarray*}
Combining all results, we arrive at the desired statement. Corollary~5.2 follows from Proposition 3.3 in Belomestny et al.~(\citeyear{BR_Levy}), because 
\begin{eqnarray*}
\P(\B_n^c)&\leq & 
 \P\bigl(\|\phi_X-\widehat \phi_X\|_{[-U_n,U_n]}
 \geq \kappa \inf_{w \in [-U_n, U_n]} |\phi_X(w)|
 \bigr)
\\
&\leq& 
 \P\bigl(\|\phi_X-\widehat \phi_X\|_{[-U_n,U_n]}
 \geq \kappa M^m U_n^{-(\beta+1) m}
 \bigr)
\lesssim\left(\sqrt{n}U_n\right)^{-2}
\end{eqnarray*}
provided that \(18U_n^{m(\beta+1)}\sqrt{\frac{\log (n U_n)}{n}}\leq \varkappa M^m . \) The last condition  is fulfilled due to our choice of the sequence \(U_n\).

\section*{Acknowledgment}
The article was prepared within the framework of the HSE University Basic Research Program.

\appendix
\section{Taylor series expansion} 

\begin{lemma}
\label{lem:taylor}
Let \(f\colon\C\to\C\) be a function that is \(k\) times differentiable (\(k=1,2,...\)) in some vicinity of a point \(a\in\C\). Then
\[
f(z) = \sum\limits_{j=0} ^{k-1} \frac{f^{(j)}(a)}{j!}(z-a)^j
+\frac{1}{(k-1)!}\E\left[(1-\tau)^{k-1}f^{(k)}(a+\tau(z-a))\right](z-a)^k,
\]
where \(\tau\) is a random variable uniformly distributed on \([0,1]\).
\end{lemma}

\section{Sufficient conditions guaranteeing~\eqref{LNN}}\label{appB}
\begin{proposition}
If the distribution of \(\tau\) is infinitely divisible, then the function
\(\phi_\Lambda\) doesn't have zeros on \(\R.\)
\end{proposition}
\begin{proof}
We have 
\(
\phi_\Lambda(u) = \mathcal{P}_\tau\bigl( 
\phi_\xi(u)
\bigr),\)
where \(\mathcal{P}_\tau(z) = \E [z^\tau]\) is the probability-generating function of \(\tau.\) Since \(\tau\) is infinitely divisible, for any \(n \in \N\), there exists a r.v. \(\tau_n\)  with probability generating function \(\mathcal{P}_n\) such that \(\mathcal{P}_\tau(z) = (\mathcal{P}_n(z))^n \; \forall z \in \C.\) Therefore, \(\Lambda\) has the same distribution as the sum on \(n\) independent copies of \(\xi_1+...+\xi_{\tau_n}\). 
\end{proof}
\begin{proposition} \label{propp}
Denote \(\mathtt{r}_k:=\P(\tau=k), \; k=1,2,...\).
 Assume that  the random variable \(\xi\) has an absolutely continuous distribution with a  finite second moment. 
 \begin{enumerate}
\item Then there exist some positive constants \(u_\circ \leq u^\circ,\) such that \(\phi_\Lambda(u) \ne 0\) for any \(|u| < u_\circ\) and \(|u| \geq u^\circ\).
\item \(u^\circ= u_\circ\) (that is, \(\phi_\Lambda(u)\) does not have real zeros) if any of the following conditions is fulfilled:
\begin{enumerate}
\item \(\mathtt{r}_m  > 1/2,\) where \(m = \argmin_{k=1,2,...} \bigl\{ \mathtt{r}_k  \ne 0 \bigr\}\);
\item the distribution of \(\xi\) is infinitely divisible with L{\'e}vy triplet \((\mu, c, \nu)\), where \(c>0\), and 
\begin{eqnarray*}
\mathtt{r}_m>\frac{1}{1+\alpha}, \quad 
\mbox{where}\qquad 
\alpha = 
\exp\Bigl\{
\frac{\pi^2}{8} \frac{
c^2}{\Var(\xi) \E [\tau]}
\Bigr\}>1.
\end{eqnarray*}
\end{enumerate}
\end{enumerate}
\end{proposition}

\begin{proof}
\textbf{1.} We have 
\begin{equation} \label{ll2}\frac{ | \phi_\Lambda(u)
|}
{|\phi_\xi(u)|^m}
 \geq
\mathtt{r}_m 
-  
\sum_{k=m+1}^\infty \mathtt{r}_k  |\phi_\xi(u)|^{k-m}.
\end{equation}
Using the Riemann--Lebesgue lemma, we get that for any \(\eps\) smaller than 1, there exists some \(u_\eps\) such that \(|\phi_\xi(u)| < \eps\) for all \(|u| > u_\eps.\) Note that for any \(\eps<1, \)
\begin{eqnarray*}
\frac{ | \phi_\Lambda(u)
|}
{|\phi_\xi(u)|^m} \geq \mathtt{r}_m - \eps \sum_{k=m+1}^\infty \mathtt{r}_k = 
\mathtt{r}_m - \eps (1- \mathtt{r}_m).
\end{eqnarray*}
Therefore, \(\phi_\Lambda(u)\) doesn't have any zeros  with absolute value larger that \(u^\circ:=u_{\eps^*}\), where \(\eps^* < \mathtt{r}_m/(1-\mathtt{r}_m).\)

On another side, Theorem~2.10.1 from Ushakov~(\citeyear{Ushakov}) yields that the characteristic function for any (not necessary infinitely divisible) r.v. \(\eta\) doesn't have zeros for \(|u| < \pi / (2 \sigma),\) where \(\sigma\) is the standard deviation of the distribution with characteristic function \(\phi_\Lambda.\)  Therefore, we can choose \(u_\circ := \min\Bigl(\pi / (2 \sigma), u^\circ\Bigr)\)  and get the required statement.\newline
\textbf{2(a).}  When  \(\mathtt{r}_m>1/2\), we have 
\begin{eqnarray*}
\frac{ | \phi_\Lambda(u)  |}
{|\phi_\xi(u)|^m}
\geq 
\mathtt{r}_m 
-  
\sum_{k=m+1}^\infty \mathtt{r}_k |\phi_\xi(u)|^{k-m}\geq 
\mathtt{r}_m - \sum_{k=m+1}^\infty \mathtt{r}_k \geq 2 \mathtt{r}_m -1 >0.
\end{eqnarray*}
\textbf{2(b).} Since \(c>0,\) we can use the  inequality 
\(
\Re \psi_\xi(u)  \leq - \frac{1}{2} u^2 c^2 
\)
to continue the line of reasoning in~\eqref{ll2}:
\begin{eqnarray} \label{ll3}\frac{ | \phi_\Lambda(u)
|}
{|\phi_\xi(u)|^m}
&\geq& 
\mathtt{r}_m 
-  
\sum_{k=m+1}^\infty \mathtt{r}_k e^{-(k-m) u^2 c^2/2}.
\end{eqnarray}
Since the function in the right-hand side monotonically increases,
it is sufficient to show that it is positive at \(
\tilde{u}= \pi/ (2\sigma)\).
 We have
\begin{eqnarray*}
\mathtt{r}_m  - \sum_{k=m+1}^\infty \mathtt{r}_k e^{- (k-m) \tilde{u} ^2 c^2 /2 }
&\geq& 
\mathtt{r}_m - \Bigl( \sum_{k=m+1}^\infty \mathtt{r}_k \Bigr)  e^{-\tilde{u} ^2 c^2 /2}\\
&\geq& \mathtt{r}_m  - (1- \mathtt{r}_m)\alpha^{-1},
\end{eqnarray*}
where the last expression is positive iff \(\mathtt{r}_m>\bigl( 1+\alpha\bigr)^{-1}.\) This observation completes the proof.
%
\end{proof}

\bibliographystyle{imsart-nameyear}
\bibliography{references}

\end{document}